\documentclass[12pt]{article}
\usepackage{curves}
\usepackage{amsmath, amssymb}
\def\mineappendix{
        \setcounter{section}{1}
        \setcounter{subsection}{0}
        \def\thesection{\Alph{section}}
        \def\sectionap{\@startsection  {section}{1}{\z@}
                        {-3.5ex plus-1ex minus-.2ex} {0ex plus.2ex}
                        {\reset@font\Large\bf  Appendix:  \, }
                        }
        }
\makeatother
\def\Proclaim #1. #2\par{\bigbreak\noindent{\sc#1.\enspace}{\it#2}\par}

\newtheorem{lemma}{Lemma}[section]
\newtheorem{corollary}[lemma]{Corollary}
\newtheorem{theorem}[lemma]{Theorem}
\newtheorem{proposition}[lemma]{Proposition}
\newtheorem{example}[lemma]{Example}
\newtheorem{definition}[lemma]{Definition}
\newtheorem{remark}[lemma]{Remark}

\title{Isometry Groups of Right Invariant Metrics in  Geometric Quantum Complexity}
\author{Xiaobo Liu\thanks{Research was partially supported by NSFC grants 12341105 and 12526302.}, \,\,\, Lei Zheng}
\date{}

\begin{document}
\maketitle

\begin{abstract}
	In this paper, we give a complete description for the full isometry groups of a class of right invariant Riemannian metrics on
the special unitary group $\mathrm{SU}(2^N)$.  These metrics have been used by physicists to study  Nielsen's geometric approach for complexities in quantum computations.
\end{abstract}

\section{Introduction}

    The symmetry of  a Riemannian manifold is described by its  isometry group. This group plays an important role
    in the study of geometric properties of the manifold. An interesting class of Riemannian manifolds is Lie groups endowed with right invariant (or equivalently left invariant) metrics. There has been lots of interests in describing the full isometry groups of such manifolds. For example, isometry groups of certain
    3 or 4 dimensional Lie groups with left invariant metrics were studied in 
    \cite{Ayad2024unimodular}, \cite{Ayad2025nonunimodular}, \cite{CR}, \cite{ha2012isometry}, \cite{Shin}, etc.
    The isometry groups of  compact connected absolutely simple real Lie groups with bi-invariant metrics are given in \cite{DP}. In general, it is quite difficult to give a complete description of the full isometry groups of high dimensional Lie groups whose metrics are not bi-invariant. In this paper, we determine the 
    isometry groups of a class of right invariant Riemannian metrics on the special unitary group $\mathrm{SU}(2^N)$ with arbitrarily large $N$.  These metrics have been used in the study of geometric quantum complexity. 
     
    In quantum computations, an algorithm involving $N$ qubits is represented by an element $U \in \mathrm{SU}(2^N)$. After choosing a set of elementary gates, which are special elements in $\mathrm{SU}(2^N)$,
    the {\it circuit complexity} of $U$ is defined to be the minimal number of elementary gates whose product is equal to $U$.
    While this definition provides an intuitive measure of the physical resources needed to implement a quantum algorithm, it is very difficult to compute circuit complexity.
    To solve this problem,
    Nielsen and his collaborators proposed a {\it geometric definition of quantum complexity} (see \cite{DN2006,gu2008quantum,N2006,nielsen2006optimal,nielsen2006quantum}). In Nielsen's approach, the complexity of $U$ is reformulated as the distance between $U$ and the identity element with respect to certain
    right invariant Riemannian metrics on $\mathrm{SU}(2^N)$. 
    Beyond quantum computations, Nielsen's geometric  approach to quantum complexity also has important applications in the study of holographic duality. In the context of the Complexity=Volume conjecture \cite{Susskind2016CV} and Complexity=Action conjecture \cite{Brown2016CA}, Nielsen's geometric approach provides natural models for the holographic complexity \cite{auzzi2021geometry,Brown2023Universality}. 
    
    A right invariant Riemannian metric on $\mathrm{SU}(2^N)$ is determined by 
    an inner product on its Lie algebra $\mathfrak{su}(2^N)$. 
    Let $\sigma_0$ be the $2 \times 2$ identity matrix and $\sigma_1, \sigma_2, \sigma_3$ be the
    standard $2\times 2$ Pauli matrices (see equation \eqref{eqn:Pauli} for a definition).
    A basis of $\mathfrak{su}(2^N)$ is given by {\it generalized Pauli matrices}
      \[ \tilde\sigma_I := - \sqrt{-1} \, \sigma_{i_1} \otimes \cdots \otimes \sigma_{i_N}, \]
    where $i_1, \ldots, i_N \in \{0, 1, 2, 3 \}$ and $I=(i_1, \ldots, i_N) \neq (0, \ldots, 0)$. 
    Each $\tilde\sigma_I$ acts naturally on $\left( \mathbb{C}^2 \right)^{\otimes N} \cong \mathbb{C}^{2^N}$, which is the Hilbert space in an $N$-qubit system.
    The {\it weight} (also called {\it locality}) of $\tilde\sigma_I$, denoted by $w(\tilde\sigma_I)$, is defined to be
    the number of $j$ such that $i_j \neq 0$. 
    In geometric quantum complexity, we need to choose a right invariant Riemannian metric $g$ on $\mathrm{SU}(2^N)$  such that
    $\tilde\sigma_I$ and $\tilde\sigma_J$ are orthogonal if $I \neq J$. 
    Such a metric is uniquely determined by positive constants
    $q_I := g(\tilde\sigma_I, \tilde\sigma_I)$, which are called the {\it penalty factor}  
    along direction $\tilde\sigma_I$. It is also common to choose $q_I = q_J$ if 
    $w(\tilde\sigma_I) = w(\tilde\sigma_J)$. In such cases, $g=g_{\bf q}$ is determined by
    $N$ positive constants ${\bf q}=(q_1, \cdots, q_N)$ such that $q_I = q_{w(\tilde\sigma_I)}$.
    We call $g_{\bf q}$ the {\it penalty metric} with penalty factors $q_1, \ldots, q_N$.    
    The main result of this paper is the following:
    
    \begin{theorem} \label{thm:mainIntro}
    Assume $N \geq 2$. Let $g$ be a penalty metric on $\mathrm{SU}(2^N)$ such that every $g$-preserving automorphism of
    $\mathrm{SU}(2^N)$  also preserves
    the subspace $\mathrm{span}_{\mathbb{R}}\{\tilde\sigma_I \in \mathfrak{su}(2^N) \mid w(\tilde\sigma_I) = 1\}$.
    Then the full isometry group of $g$ is isomorphic to 
    \begin{equation} \label{eqn:IsometryIsom}
    \mathrm{SU}(2^N)\rtimes \{ (\mathrm{SO}(3)^N \rtimes \mathbb{Z}_2) \rtimes S_N \}.
    \end{equation}
    \end{theorem}
    
    This theorem follows from Corollary \ref{cor:mainResult}. We would like to mention that most commonly used penalty metrics in geometric quantum complexity
    satisfy the condition in Theorem \ref{thm:mainIntro}. For example, the {\it cliff metrics}
    are the penalty metrics with penalty factors $q_1=q_2=1$ and $q_3= \cdots = q_N > 1$. 
    In \cite{nielsen2006optimal}, Nielsen and Dowling proved that the geometric quantum complexities defined by the cliff metrics are polynomially equivalent to the approximate circuit complexity when $q_3$ is sufficiently large. 
    In Proposition \ref{cliff_metric_preserve_locality}, we will show that cliff metrics satisfy the condition in Theorem \ref{thm:mainIntro}. In Proposition \ref{condition_isotropy}, we will show that all penalty metrics with
    $q_i \neq q_1$ for $i>1$  satisfy the condition in Theorem \ref{thm:mainIntro}. Such metrics include
    the binomial metrics and exponential metrics defined in Examples \ref{ex:Binomial} and \ref{ex:exponential}.
    In \cite{brown2022}, Brown showed that the geometric quantum  complexities defined by the
    binomial and exponential metrics with appropriate choices of parameters  
    are also polynomially equivalent to the approximate circuit complexity. 
    
    Theorem \ref{thm:mainIntro} implies that the full isometry group of the penalty metric $g$ is generated by the following four types of isometries:
    
    \begin{itemize}
    
    \item[(i)] Right translations by elements in $\mathrm{SU}(2^N)$ are isometries. 
    This gives the first factor in expression~\eqref{eqn:IsometryIsom}.
    
    \item[(ii)] Let $\mathrm{SU}(2)^{\otimes N}$ be the subgroup of $\mathrm{SU}(2^N)$ consisting of
    all $U=U_1 \otimes \cdots \otimes U_N$ with $U_i \in \mathrm{SU}(2)$. The action of $U$ on
     $v=v_1 \otimes \cdots \otimes  v_N \in (\mathbb{C}^2)^{\otimes N}$ is defined by
    \[
    U \cdot v \, :=  \, (U_1 \cdot v_1) \otimes \cdots \otimes ( U_N \cdot  v_N).
    \]
    The inner automorphisms on $\mathrm{SU}(2^N)$ given by adjoint actions of elements in $\mathrm{SU}(2)^{\otimes N}$ are also isometries. In Lemma \ref{lem:SU2SO3N}, we will show that
    the group of all such isometries is isomorphic to $\mathrm{SO}(3)^N$, which gives the second factor
    in expression \eqref{eqn:IsometryIsom}.
    
    \item[(iii)] Taking complex conjugation for elements in $\mathrm{SU}(2^N)$ is also an isometry.
    The group generated by this isometry gives the $\mathbb{Z}_2$ factor in expression \eqref{eqn:IsometryIsom}.
    
    \item[(iv)] For any permutation $\tau \in S_N$, we can define a linear map $\mathfrak{su}(2^N) \to \mathfrak{su}(2^N)$ which sends $\tilde\sigma_{(i_1, \ldots, i_N)}$ to
        $\tilde\sigma_{(i_{\tau(1)}, \ldots, i_{\tau(N)})}$. In Lemma \ref{lem:UijSN} and equation \eqref{eqn:P=SN}, we will show that such linear maps are generated by differentials of adjoint actions
        of certain elements $U_{kl} \in \mathrm{SU}(2^N)$ 
        defined by equation \eqref{SWAP}. Adjoint actions by $U_{kl}$
        are isometries. They generate the $S_N$ factor in expression \eqref{eqn:IsometryIsom}.

    \end{itemize}

    Isometries of types (i)-(iii) were constructed by Nielsen in his foundational work on the geometric quantum complexity \cite[Section IV. A]{N2006}. But the group structure of the full isometry group was not given in \cite{N2006}.  Nielsen also pointed out that "obtaining a complete classification
    of the isometries is an interesting problem".     
     Our results in this paper show that together with isometries in type (iv), we obtain a complete
     classification of all isometries for the commonly used metrics in geometric quantum complexity.
    Moreover, the group structure of the full isometry group is also clear from Theorem \ref{thm:mainIntro}.

    This paper is organized as follows: In Section 2, we review basic facts about geometry of general right invariant metrics on Lie groups and penalty metrics appeared in geometric quantum complexity. 
     In Section 3, we calculate the identity component of the isotropy group. 
     In Section 4, we narrow down possible elements in the isometry group 
      by calculating the normalizer of the identity component of the isotropy group realized as a subgroup in
      the orthogonal transformation group on $\mathfrak{su}(2^N)$. The proof of Theorem \ref{thm:mainIntro}
     will be completed in Section 5.

\section{Preliminaries}

    \subsection{Isometry groups of Lie groups with right invariant metrics}
    
    We first introduce basic notation and  translate some known results about left invariant metrics on
     Lie groups to right invariant metrics.

    Let $G$ be a Lie group. For any $x \in G$, let $R_x$ and $L_x$ be the right and left translations on $G$ defined by
    $R_x(y) := yx$ and $L_x(y) := xy$ for all $y \in G$.
    A Riemannian metric $g$ on $G$ is \textit{right} (or \textit{left}) \textit{invariant} if $R_x^*g = g$
    (or $L_x^*g = g$)  for all $x \in G$.
    Left invariant metrics have been extensively studied in mathematical literature (see, for example, \cite{M1976}).  However, physicists prefer to use right invariant metrics in the study of quantum complexity (see, for example, \cite{DN2006} and \cite{Brown2023Universality}). The following lemma can be used to translate standard results for left invariant metrics to results for right invariant metrics.

\begin{lemma}\label{left_right_relation}
    Let $G$ be a Lie group, and $\Theta$ the inversion map defined by $\Theta(x) := x^{-1}$ for all $x \in G$. Then a metric $g$ on $G$ is right invariant if and only if the pullback metric $\Theta^{*}g$ is left invariant. Furthermore, there is a
    group isomorphism
    \[
    \begin{array}{lccc}
    \Psi: & \mathrm{Isom}(G,\Theta^*g) & \to & \mathrm{Isom}(G,g), \\
                  & \psi & \mapsto & \Theta \circ \psi \circ \Theta,
    \end{array}
    \]
    where $\mathrm{Isom}(G,g)$ is the group of all isometries of the Riemannian manifold $(G,g)$.
\end{lemma}

    \noindent
    {\bf Proof}:
    First observe that for every $y \in G$,
    \begin{equation} \label{eqn:LRTrans}
    \Theta \circ L_y = R_{y^{-1}}\circ \Theta.
    \end{equation}
    If $g$ is right invariant, then
    \[ L_y^*(\Theta^*g) = (\Theta \circ L_y)^* g = (R_{y^{-1}} \circ \Theta)^*g = \Theta^*(R_{y^{-1}}^*g) = \Theta^*g. \]
    So $\Theta^*g$ is left invariant. On the other hand, if $\Theta^*g$ is left invariant, then
    \[ R_y^* g = (\Theta \circ L_{y^{-1}} \circ \Theta)^* g = \Theta^* L_{y^{-1}}^* (\Theta^* g)
        = \Theta^* (\Theta^* g) = (\Theta^2)^*g  = g. \]
    Hence $g$ is right invariant.

    Moreover, since $\Theta = \Theta^{-1}$ is an isometry from $(G,\Theta^*g)$ to $(G,g)$ , it follows that $\Psi$
    is a group isomorphism from $\mathrm{Isom}(G,\Theta^*g)$ to $\mathrm{Isom}(G,g)$.
    $\Box$

     In this paper, we will study $\mathrm{Isom}(G,g)$ where $G$ is a connected Lie group equipped with a right invariant metric $g$.  Let
    \[ R(G):=\{R_x \mid x \in G\}, \hspace{20pt} L(G):=\{L_x \mid x \in G\}.\]
    Since $g$ is right invariant, we have $R(G) \subseteq \mathrm{Isom}(G,g)$.
    Let $\mathrm{Isom}_e(G,g)$ be the isotropy subgroup of $\mathrm{Isom}(G,g)$ at the identity element $e \in G$.
    Then we have
    \begin{equation}\label{eq:isom_decomposition}
    \mathrm{Isom}(G,g) = R(G) \cdot \mathrm{Isom}_e(G,g),  \hspace{20pt}
    R(G) \cap \mathrm{Isom}_e(G,g)=\{\mathrm{Id}\},
    \end{equation}
    where $\mathrm{Id}$ is the identity map from $G$ to itself.
    In general, the group structures of $R(G)$ and $\mathrm{Isom}_e(G,g)$ do not completely determine the group structure of $\mathrm{Isom}(G,g)$. But if $R(G)$ is a normal subgroup of $\mathrm{Isom}(G,g)$, then $\mathrm{Isom}(G,g)$ is isomorphic to the semidirect product of $R(G)$ and $\mathrm{Isom}_e(G,g)$ (see, for example,   \cite[Exercise 12 on page 76]{L2002}). In this case, we write
    \begin{equation} \label{eqn:semiprod}
    \mathrm{Isom}(G,g) =  R(G) \rtimes \mathrm{Isom}_e(G,g) \cong G \rtimes \mathrm{Isom}_e(G,g).
    \end{equation}
    The following lemma gives a criterion for $R(G)$ to be a normal subgroup of $\mathrm{Isom}(G,g)$. It follows from
    Lemma~\ref{left_right_relation} and analogous results for left invariant metrics given in \cite[Lemma~2.2]{ha2012isometry}
    and \cite[Lemma~1.1]{Shin}.

\begin{lemma}\label{normalsubgp}
    Let $G$ be a Lie group endowed with a right invariant metric $g$. Then $R(G)$ is a normal subgroup of $\mathrm{Isom}(G,g)$ if and only if $\mathrm{Isom}_e(G,g) \subseteq \mathrm{Aut}(G)$, where $\mathrm{Aut}(G)$ is the group of automorphisms of $G$. Furthermore, in this case, equation \eqref{eqn:semiprod} holds.
\end{lemma}

\begin{remark}
Define $\mathrm{Aut}(G,g) := \mathrm{Aut}(G) \cap \mathrm{Isom}(G,g)$.
Then by Lemma \ref{normalsubgp}, $R(G)$ is a normal subgroup of $\mathrm{Isom}(G,g)$ if and only if
$\mathrm{Isom}_e(G,g) = \mathrm{Aut}(G, g)$.
\end{remark}

For any Lie group $K$, we will use $K_0$ to denote the identity component of $K$.
The following proposition is a right invariant analogue of Ochiai and Takahashi's result \cite[Theorem~1]{Ochiai1976}.
\begin{proposition}\label{simpleLiegp}
    Let $G$ be a compact connected simple Lie group equipped with a right invariant metric  $g$. Then
    \[
    \mathrm{Isom}(G,g)_0 \subseteq R(G)L(G).
    \]
\end{proposition}

    \noindent
    {\bf Proof}: Let $\Theta$ be the inversion map on $G$. By Lemma \ref{left_right_relation},
    $\Theta^* g$ is a left invariant metric on $G$. Hence
    \cite[Theorem~1]{Ochiai1976} implies that   $\mathrm{Isom}(G, \Theta^* g)_0 \subseteq L(G)R(G)$.
    By equation \eqref{eqn:LRTrans}, $\Theta \, L(G) \, \Theta = R(G)$ and  $\Theta \, R(G) \, \Theta = L(G)$.
    So by  Lemma \ref{left_right_relation}, we have
    \[ \mathrm{Isom}(G,g)_0 \, = \, \Theta \, \mathrm{Isom}(G, \Theta^* g)_0 \, \Theta
        \, \subseteq \, (\Theta \, L(G) \, \Theta) \, (\Theta \, R(G) \, \Theta ) \, = \, R(G)L(G).\]
      $\Box$

    Proposition \ref{simpleLiegp} gives an important tool to compute $\mathrm{Isom}(G,g)_0$. To apply this proposition,
      we first clarify the relationship between $\mathrm{Isom}_e(G,g)_0$ and the isotropy subgroup of $\mathrm{Isom}(G,g)_0$ at $e$.

    \begin{lemma}\label{identity_component_isotropy_right}
    Let $G$ be a connected Lie group equipped with a right invariant Riemannian metric $g$. Then
    \[
    \mathrm{Isom}_e(G,g)_0 = [\mathrm{Isom}(G,g)_0]_e,
    \]
    where the left-hand side denotes the identity component of $\mathrm{Isom}_e(G,g)$, and the right-hand side denotes the isotropy subgroup of $\mathrm{Isom}(G,g)_0$ at $e$.
    \end{lemma}

    \noindent
    {\bf Proof}:
    Define a map $\Xi: R(G) \times \mathrm{Isom}_e(G,g) \longrightarrow \mathrm{Isom}(G,g)$  by
     $\Xi(R_x, \phi) = R_x \circ \phi$
    for $x \in G$ and $\phi \in \mathrm{Isom}_e(G,g)$.
    By equation \eqref{eq:isom_decomposition}, $\Xi$ is a bijection with inverse given by
     $\Xi^{-1}(\psi) = (R_{\psi(e)}, R_{\psi(e)^{-1}} \circ \psi)$ for $\psi \in \mathrm{Isom}(G,g)$.
     It is obvious that both $\Xi$ and $\Xi^{-1}$
    are smooth maps. Hence $\Xi$ is a diffeomorphism from $R(G) \times \mathrm{Isom}_e(G,g)$ to $\mathrm{Isom}(G,g)$.
     Since $R(G) \cong G$ is connected,
    we have
    \[
    \mathrm{Isom}(G,g)_0 = R(G) \cdot  \mathrm{Isom}_e(G,g)_0.
    \]
    Hence the isotropy group of $\mathrm{Isom}(G,g)_0$ at $e$ is equal to $\mathrm{Isom}_e(G,g)_0$.
    This completes the proof of the lemma.
    $\Box$

    An immediate consequence of Proposition \ref{simpleLiegp}, Lemma \ref{identity_component_isotropy_right},
    and an argument similar to Lemma \ref{normalsubgp} (see also \cite[Theorem 2]{Ochiai1976}) is the following: 
    \begin{corollary} \label{cor:simple-component}
    If $G$ is a compact connected simple Lie group equipped with a right invariant Riemannian metric $g$, then
    \begin{equation}\label{equ:isometry_automorphism}
    \mathrm{Isom}_e(G,g)_0 = \mathrm{Aut}(G,g)_0.
    \end{equation}
    Moreover, $R(G)$ is a normal subgroup of $\mathrm{Isom}(G,g)_0$ and
    \[ \mathrm{Isom}(G,g)_0 = R(G) \rtimes \mathrm{Isom}_e(G,g)_0. \]
    \end{corollary}

By equation \eqref{equ:isometry_automorphism}, to study $\mathrm{Isom}_e(G,g)_0$, it suffices to study automorphisms of
$G$ which are also isometries with respect to  the metric $g$.
The following lemma is the right invariant analogue of a well known fact for left invariant metrics.
    \begin{lemma}\label{automorphism_isometry}
    Let $G$ be a Lie group equipped with a right invariant Riemannian metric $g$. Then $\varphi \in \mathrm{Aut}(G)$ is an isometry of $(G, g)$ if and only if its differential at the identity, $\varphi_{*e}: (T_eG, g|_e) \to (T_eG, g|_e)$, is a linear isometry.
    \end{lemma}

    \noindent
    {\bf Proof}:
    Since $\varphi$ is a group automorphism, we have
    $\varphi \circ R_x = R_{\varphi(x)} \circ \varphi $
    for all $x \in G$.
    Taking the differential of both sides at the identity $e$, we obtain
    $ \varphi_{*x}  = (R_{\varphi(x)})_{*e} \circ \varphi_{*e} \circ (R_{x})_{*e}^{-1} $.
    If $g$ is right invariant, then for every $x \in G$, $\varphi_{*x}$ is a linear isometry if and only if
    $\varphi_{*e}$ is a linear isometry. The lemma is thus proved.
    $\Box$

\begin{remark}
After choosing a basis of $T_eG$, we can represent $\varphi_{*e}$ by a matrix $[\varphi]$ and represent
 $g|_e$ by a matrix $A$.
Lemma \ref{automorphism_isometry} reduces the computation for $\mathrm{Aut}(G,g)$ to solving a matrix equation:
\begin{equation}\label{tangent_equation}
    [\varphi]^T A [\varphi] = A.
\end{equation}
For Lie groups of low dimensions, this equation might be solved directly.
For example,  this method was used in \cite{ha2012isometry, Ayad2024unimodular, Ayad2025nonunimodular} to compute
isometry groups of left invariant metrics on some $3$ and $4$ dimensional Lie groups. In this paper, we will compute
isometry groups of Lie groups with arbitrarily large dimension. In such cases,  it is quite difficult to solve
equation \eqref{tangent_equation}.
Instead, we will use the fact that isometries preserve the curvature tensor to study ${\rm Isom}(G, g)$. This method was used in
\cite{CR} to study isometry groups of 3-dimensional Lie groups.
\end{remark}

In this paper, we will focus on the special unitary group ${\rm SU}(n)$. It is well known that the automorphism group of this Lie group is given by
    \begin{equation} \label{eqn:AutSU(n)}
    \mathrm{Aut}(\mathrm{SU}(n)) \cong
                \begin{cases}\mathrm{PSU}(2) \cong \mathrm{SO}(3), & {\rm if \,\,\,} n = 2, \\
                            \mathrm{PSU}(n) \rtimes \mathbb{Z}_2, & {\rm if \,\,\,} n \geq 3, \end{cases}
    \end{equation}
    \noindent where $\mathrm{PSU}(n)$ is the quotient of $\mathrm{SU}(n)$ by its center.
    In fact, $\mathrm{PSU}(n)$ is isomorphic to the inner automorphism group of $\mathrm{SU}(n)$, and
     $\mathbb{Z}_2$ is isomorphic
     to the outer automorphism group of $\mathrm{SU}(n)$ which is generated by the complex conjugation. Equation \eqref{eqn:AutSU(n)} follows from the discussion of automorphisms of  ${\rm SU}(n)$
    in \cite[Page 102]{Loos} and a general result about automorphism groups of compact Lie groups in \cite[Theorem 6.73]{HM}.

    \subsection{Geometric Quantum  Complexity}

    In this subsection, we introduce some basic notions from quantum computation and the standard models used in geometric quantum complexity. We refer to the book \cite{NC} for basic terminologies and facts in quantum computation.

    First, the Hilbert space of a single qubit is $\mathbb{C}^2$. The Lie group ${\rm SU}(2)$ acts on this space via matrix
    multiplications. A standard basis of the Lie algebra $\mathfrak{su}(2)$ is given by $\{ - i \sigma_j \mid j=1, 2, 3\}$, where
    $i = \sqrt{-1}$ and  $\sigma_j$ are standard {\it Pauli matrices} defined by
    \begin{equation} \label{eqn:Pauli}
    \sigma_{1} = \begin{bmatrix}0 & 1 \\ 1 & 0 \end{bmatrix},\quad
    \sigma_{2} = \begin{bmatrix}0 & -i\\ i & 0\end{bmatrix},\quad
    \sigma_{3} = \begin{bmatrix}1 & 0\\ 0 & -1 \end{bmatrix}.
    \end{equation}
    For convenience, we also set
    $\sigma_{0}$ to be the $2 \times 2$ identity matrix.

    For any positive integer $N$, the Hilbert space for $N$ qubits is the $N$-fold tensor product  $(\mathbb{C}^2)^{\otimes N} \cong \mathbb{C}^{2^N}$. An algorithm in quantum computation can be represented by an element in the Lie group $\mathrm{SU}(2^N)$ which acts on $\mathbb{C}^{2^N}$ by matrix multiplication.
    For any $2 \times 2$ complex matrices $A_1, \ldots, A_N$, we define a linear operator
    $A_1 \otimes \cdots \otimes A_N$ which acts on $(\mathbb{C}^2)^{\otimes N}$ by
    \[ (A_1 \otimes \cdots \otimes A_N) \, \cdot \, (v_1 \otimes \cdots \otimes v_N)
        = (A_1 \cdot v_1) \otimes \cdots \otimes (A_N \cdot v_N) \]
    for all $v_1, \ldots, v_N \in \mathbb{C}^2$. Let
    \begin{equation} \label{eqn:I_N}
    \mathcal{I}_N := \{ (i_1, \dots, i_N) \mid i_1, \ldots, i_N \in \{0,1,2,3\} \}, \hspace{20pt}
    \mathcal{I}_N^0 := \mathcal{I}_N \setminus \{(0,  \ldots, 0)\}.
    \end{equation}
    For any $I =(i_1, \ldots, i_N) \in \mathcal{I}_N$, define
    \[ \sigma_I := \sigma_{i_1} \otimes \cdots \otimes \sigma_{i_N}. \]
    Matrices $\sigma_I$ are called  the \textit{generalized Pauli matrices}.
    Commutators $[\sigma_I,\sigma_J] := \sigma_I\sigma_J - \sigma_J\sigma_I$ of such matrices can be computed by the following:

    \begin{lemma}\label{Paulimainthm}
    Fix a positive integer $N$. For $I = (i_1, \ldots, i_N), J = (j_1, \ldots, j_N) \in \mathcal{I}_N$, let $s(I, J)$ be the number of $k$ such that
    $[\sigma_{i_k},\sigma_{j_k}] \neq 0$. Then
    \begin{equation} \label{eqn:bracketGP}
    [\sigma_I,\sigma_J] = \left( 1 - (-1)^{s(I, J)} \right) \sigma_I\sigma_J.
    \end{equation}
    Furthermore, if $[\sigma_I,\sigma_J] \neq 0$, there must exist $K \in \mathcal{I}_N^0$ and $\epsilon_{IJ} \in \{1,-1\}$ such that
    \begin{equation} \label{eqn:bracketGPnz}
    [\sigma_I, \sigma_J]  = 2 \epsilon_{IJ} \, i\sigma_K, \quad
    [\sigma_J, \sigma_K] = 2 \epsilon_{IJ} \, i\sigma_I, \quad
    [\sigma_K, \sigma_I] = 2 \epsilon_{IJ} \, i\sigma_J.
    \end{equation}
    \end{lemma}

    \noindent
    {\bf Proof}:
    First, it is straightforward to check that $ \sigma_{0} \sigma_{j} = \sigma_{j} \sigma_{0} = \sigma_{j}$ and
    \begin{equation} \label{eqn:ProdP}
     \sigma_j^2 = \sigma_0, \hspace{20pt}
        \sigma_{\tau(1)} \sigma_{\tau(2)} = - \sigma_{\tau(2)} \sigma_{\tau(1)} = i \sigma_{\tau(3)},
    \end{equation}
     where $j \in \{0, 1, 2, 3\}$ and $\tau$ is any cyclic permutation of $(1, 2, 3)$.
    So standard Pauli matrices either commute or anti-commute.
    Consequently, we have
    \begin{equation} \label{eqn:sigma-IJ}
    \sigma_I \sigma_J = (\sigma_{i_1} \sigma_{j_1}) \otimes \cdots \otimes (\sigma_{i_N} \sigma_{j_N})
        = (-1)^{s(I, J)} \sigma_J \sigma_I
    \end{equation}
    for any $I=(i_1, \ldots, i_N), J=(j_1, \ldots, j_N) \in \mathcal{I}_N$.
    This proves equation \eqref{eqn:bracketGP}.

    By equation \eqref{eqn:ProdP},  for each $1 \leq m \leq N$, there exists a unique
    $k_m \in \{0, 1, 2, 3\}$ and $a_m \in \{1, i,  -i\}$ such that
    \begin{equation} \label{eqn:ijk-m}
        \sigma_{i_m} \sigma_{j_m} = a_m \,  \sigma_{k_m}.
    \end{equation}
    In fact, $a_m = \pm i$ if and only if $[\sigma_{i_m}, \sigma_{j_m}] \neq 0$, which
    occurs only when $i_m$, $j_m$ and $0$ are mutually distinct. In this case, we also have $k_m \neq 0$.
    The number of $m$ such that $a_m = \pm i$ is precisely equal to $s(I, J)$.
    It is also straightforward to check that
    equation \eqref{eqn:ijk-m} still holds for all cyclic permutations of $(i_m, j_m, k_m)$.
    Let $K=(k_1, \ldots, k_N)$. Then we have
    \begin{equation} \label{eqn:IJK}
     \sigma_I \sigma_J = \, \left( \prod_{m=1}^N a_m \right) \, \sigma_K,
    \end{equation}
    and this equation also holds for all cyclic permutations of $(I, J, K)$.

    If $[\sigma_I,\sigma_J] \neq 0$, then by equation \eqref{eqn:bracketGP}, $s(I, J)$ must be odd and
    $K \in \mathcal{I}_N^0$. In this case,
    $\prod_{m=1}^N a_m = \epsilon_{IJ} \, i $ for some $\epsilon_{IJ} \in \{1, -1 \}$. Hence we have
    \begin{equation} \label{eqn:IJKodd}
     \sigma_I \sigma_J = - \sigma_J \sigma_I = \, \epsilon_{IJ} \, i \, \sigma_K,
    \end{equation}
    which also holds for all cyclic permutations of $(I, J, K)$ and
    $\epsilon_{IJ}=\epsilon_{JK}=\epsilon_{KI}$.
    This equation implies equation \eqref{eqn:bracketGPnz}. The lemma is thus proved.
    $\Box$

    \begin{remark} \label{rem:IJK}
    Equation \eqref{eqn:sigma-IJ} is just \cite[Equation (2.5)]{auzzi2021geometry}. It follows from equations
    \eqref{eqn:bracketGPnz} and \eqref{eqn:IJK} that both  $[\sigma_I, \sigma_J]$ and $\sigma_I \sigma_J$  are always proportional to some
    generalized Pauli matrices (see also \cite[Page 4]{auzzi2021geometry}).
    Equation \eqref{eqn:bracketGPnz} also implies that
    \begin{equation} \label{eqn:IIJ=4J}
    [\sigma_I, [\sigma_I, \sigma_J]] = 4 \sigma_J
    \end{equation}
    if $[\sigma_I, \sigma_J] \neq 0$ (see also \cite[arXiv version, Section 1.5]{Brown2023}).
    \end{remark}

    If $I=(i_1, \ldots, i_N) \in \mathcal{I}_N$ has exactly $k$ non-zero components,
    then  we define
    \begin{equation} \label{eqn:locality}
            w(c \sigma_I) := k
    \end{equation}
    for any $0 \neq c \in \mathbb{C}$.
    The integer $w(c \sigma_I)$ is called the {\it weight} (or {\it locality}) of $c \sigma_I$.
    If $w(c \sigma_I)=k$, then $c \sigma_I$ acts nontrivially on exactly
    $k$ components in the space $(\mathbb{C}^2)^{\otimes N}$.

    A basis for the Lie algebra $\mathfrak{su}(2^N)$ is given by
    \begin{equation}
    \{\tilde{\sigma}_I := -i\sigma_I \mid I \in \mathcal{I}_N^0\}.
    \end{equation}
    In quantum computation, unitary operators of the form $\exp(t \tilde{\sigma}_I)$ with $w(\sigma_I) \in \{1, 2\}$ and
    $0 \leq t \leq 1$ are typically easier to implement. These operators are considered as {\it elementary gates} in quantum computation. The {\it circuit complexity} of an operator $U \in {\rm SU}(2^N)$ is defined to be the minimum number of elementary
    gates whose product is equal to $U$ (cf. \cite{N2006}). The circuit complexity is usually hard to compute. In \cite{N2006}, Nielsen proposed a geometric approach to approximate
    the circuit complexity of $U$ by the distance between $U$ and the identity element $e \in {\rm SU}(2^N)$ after choosing an appropriate right invariant Riemannian metric $g$ on ${\rm SU}(2^N)$.
    The distance between $U$ and $e$ is then called the {\it geometric quantum complexity} of $U$.
    The Riemannian metric $g$ is usually chosen to be of the following form:

    \begin{definition}[Penalty metrics] \label{def:Penalty}
    Given ${\bf q}=(q_1, \dots, q_N)$ with all $q_i > 0$. Let $g_{\bf q}$ be the right invariant Riemannian metric on
    $\mathrm{SU}(2^N)$ whose restriction to $T_e {\rm SU}(2^N) \cong \mathfrak{su}(2^N)$ is given by
    \[
    g_{\bf q}(\tilde{\sigma}_I, \tilde\sigma_J)  =
                 q_{w(\tilde\sigma_I)} \delta_{IJ}
    \]
    for all $I, J \in \mathcal{I}_N^0$, where  $\delta_{IJ} = 1$ if $I=J$ and $\delta_{IJ} = 0$ if $I \neq J$.
    The metric $g_{\bf q}$ is called the \textit{penalty metric} with penalty factors $q_1, \cdots, q_N$.
    \end{definition}


     For convenience, we also define $q_I := q_k$ if $w(\sigma_I) = k$ for any $I \in \mathcal{I}_N^0$.
     Then we have

    \begin{lemma} \label{lem:AdInv}
    For any $I, J, K \in \mathcal{I}_N^0$, either
    $g_{\bf q}([\tilde{\sigma}_I, \tilde\sigma_J], \tilde{\sigma}_K)
        = g_{\bf q}(\tilde\sigma_J,  [\tilde{\sigma}_I, \tilde{\sigma}_K]) = 0$, or
     \begin{equation} \label{eqn:adInv}
      g_{\bf q}([\tilde{\sigma}_I, \tilde\sigma_J], \tilde{\sigma}_K)
        + g_{\bf q}(\tilde\sigma_J,  [\tilde{\sigma}_I, \tilde{\sigma}_K]) = \pm 2  (q_K - q_J).
      \end{equation}
     In particular  $g_{\bf q}$ is bi-invariant  if and only if $q_1= \cdots = q_N$.
    \end{lemma}

    \noindent
    {\bf Proof}: Equation \eqref{eqn:adInv} follows directly from equation \eqref{eqn:bracketGPnz}.
    The metric $g_{\bf q}$ is bi-invariant if and only if
    $g_{\bf q}([\tilde{\sigma}_I, \tilde\sigma_J], \tilde{\sigma}_K)
        + g_{\bf q}(\tilde\sigma_J,  [\tilde{\sigma}_I, \tilde{\sigma}_K]) = 0$
    for all $I, J, K \in \mathcal{I}_N^0$. Hence
     equation~\eqref{eqn:adInv} implies $g_{\bf q}$ is bi-invariant
     if $q_1= \cdots = q_N$.

    On the other hand, if penalty factors are not all equal, there must exist $1 \leq t \leq N-1$ such that
    $q_t \neq q_{t+1}$.
    Let $J=(j_1, \ldots, j_N)$ with $j_r= 1$ for $r \leq t$ and $j_r= 0$ for $t < r \leq N$,
    $I=(i_1, \ldots, i_N)$ with $i_t= 2$, $i_{t+1}=1$, and $i_r = 0$ for $r \notin \{t, t+1\}$,
     and $K=(k_1, \ldots, k_N)$ with $k_r= 1$ for $r < t$, $k_t= 3$, $k_{t+1}=1$, $k_r=0$ for $t+1 < r \leq N$.
    Then
    $[\tilde\sigma_I, \tilde\sigma_J] = -2 \tilde\sigma_K$
    and $[\tilde\sigma_I, \tilde\sigma_K] = 2 \tilde\sigma_J$.
    Since $w(\tilde\sigma_J)=t$ and $w(\tilde\sigma_K)=t+1$, we have
    \[ g_{\bf q}([\tilde{\sigma}_I, \tilde\sigma_J], \tilde{\sigma}_K)
        + g_{\bf q}(\tilde\sigma_J,  [\tilde{\sigma}_I, \tilde{\sigma}_K]) = 2 (q_{t} - q_{t+1}) \neq 0. \]
    Therefore $g_{\bf q}$ is not bi-invariant. The lemma is thus proved.
    $\Box$

    Typical examples of penalty metrics
    are the following:

    \begin{example}[Cliff metric] \label{ex:Cliff}
    Given $q > 1$. The penalty metric $g_{\bf q}$ with
    \[
    q_1=q_2=1, \,\,\, \text{and} \,\,\, q_k=q \,\,\, \text{for} \,\,\, \text{all} \,\,\, k \geq 3
    \]
    is called the \textit{cliff metric} with penalty factor $q$. This metric is also denoted by $g_q$.
    \end{example}

    \begin{example}[Binomial metric] \label{ex:Binomial}
    Given $\alpha > 0$. The penalty metric $g_{\bf q}$ with
    \[
       q_k= \left(\binom{N}{k} 3^k\right)^\alpha
    \]
    for all $k=1, \ldots , N$   is called the \textit{binomial metric}.
    \end{example}

    \begin{example}[Exponential metric] \label{ex:exponential}
    Given $x > 1$. The penalty metric $g_{\bf q}$ with
    \[
     q_k= x^{2k}
    \]
    for all $k=1, \ldots , N$ is called the \textit{exponential metric}.
    \end{example}

    The cliff metric $g_q$ was first introduced by Nielsen in \cite{N2006}.  It was proven in \cite{nielsen2006quantum} that the geometric quantum complexity with respect to this metric is polynomially equivalent to the approximate circuit complexity if $q$ is sufficiently large. The reason for choosing cliff metric is highly intuitive: when $q$ is sufficiently large (typically $q > 4^N$), it imposes a severe penalty on directions with weights $\geq  3$. This forces the tangent direction of the shortest path from $e$ to $U \in {\rm SU}(2^N)$ to be  close to directions with weights equal to $1$ and $2$. The term "cliff metric" originates precisely from this abrupt, cliff-like jump in the  assignment of penalty factors.

    However, the choice of the penalty metric in quantum geometric complexity is neither unique nor universally standardized. For example, in \cite{brown2022}, Brown investigated several penalty metrics with respect to which the geometric quantum  complexity is also polynomially equivalent to the approximate circuit complexity. Such metrics include the binomial metric and
    exponential metric.

    To study geometric quantum complexity, we need to compute the curvature tensor of penalty metrics $g$:
    \[
    R(X,Y,Z,W) := g(\nabla_X\nabla_YZ - \nabla_Y\nabla_XZ - \nabla_{[X,Y]}Z , W),
    \]
    where $X,Y,Z,W$ are smooth vector fields on $\mathrm{SU}(2^N)$ and $\nabla$ is the Levi-Civita connection of $g$.
    Milnor has computed the
    curvature tensor of left invariant metrics in \cite{M1976}. Using Milnor's result, Dowling and Nielsen  computed the curvature tensor of penalty metrics in \cite{DN2006} (see also \cite{auzzi2021geometry}).
    To state this result, we first introduce some notation.
    Note that whenever $[\tilde\sigma_I,\tilde\sigma_J] \neq 0$, it is a scalar multiple of some $\tilde\sigma_S$ (see Lemma \ref{Paulimainthm}). Define
    \begin{equation}
    q_{[I,J]} :=
    \begin{cases}
    q_S, & {\rm if} \,\,\, [\tilde\sigma_I,\tilde\sigma_J] \neq 0,\\
    1, & {\rm if} \,\,\, [\tilde\sigma_I,\tilde\sigma_J] = 0,
    \end{cases}
    \end{equation}
    and
    \begin{equation}
    c_{I,J}:= \frac{1}{2}\left(1 + \frac{q_J - q_I}{q_{[I,J]}}\right), \hspace{20pt}
    c_{[I,J],K} :=
    \begin{cases}
    c_{S,K},& {\rm if} \,\,\, [\tilde\sigma_I,\tilde\sigma_J] \neq 0, \\
    1, & {\rm if} \,\,\, [\tilde\sigma_I,\tilde\sigma_J] = 0.
    \end{cases}
    \end{equation}
    Then we have

    \begin{proposition}[{\cite[Appendix A]{DN2006}, \cite[Section II C]{auzzi2021geometry}}]
    Let $g = g_{\bf q}$ be the penalty metric on $\mathrm{SU}(2^N)$ with penalty factors ${\bf q}=(q_1, \cdots, q_N)$.
    For any $I \in \mathcal{I}_N^0$, let $X_I$ be the right invariant vector field on $\mathrm{SU}(2^N)$ satisfying $X_I \mid_e = \tilde\sigma_I$. Then the curvature tensor of $g$ is given by
    \begin{equation}\label{curvature_tensor_formula}
    \begin{aligned}
    R_{IJKL}&:= R(X_I,X_J,X_K,X_L)\\
    &= c_{I,K} \, c_{J,L} \, g([\tilde\sigma_I, \tilde\sigma_K],[\tilde\sigma_J, \tilde\sigma_L]) \\
    &\phantom{=} - c_{J,K} \, c_{I,L} \, g([\tilde\sigma_J, \tilde\sigma_K],[\tilde\sigma_I, \tilde\sigma_L])\\
    &\phantom{=}-c_{[I,J],K} \, g([[\tilde\sigma_I,\tilde\sigma_J],\tilde\sigma_K],\tilde\sigma_L).
    \end{aligned}
    \end{equation}
    \end{proposition}

\section{The identity component of the isometry group}

    In this section, we will compute the identity component of the isometry group of a penalty metric $g_{\bf q}$ on
    $\mathrm{SU}(2^N)$ given in Definition \ref{def:Penalty}.

    We first introduce some notation.
    For any Lie group $G$, we will denote the adjoint action of $G$ on itself by $\mathcal{A}{\rm d}$ and the adjoint action
    of $G$ on its Lie algebra by ${\rm Ad}$. For any $x \in G$, $\mathcal{A}{\rm d}_x$ is the inner automorphism on $G$ defined by $\mathcal{A}{\rm d}_x (y) := x y x^{-1}$ for all $y \in G$.
    Let ${\rm Inn}(G) := \{ \mathcal{A}{\rm d}_x \mid x \in G\}$.
    If $G$ is a subgroup of $\mathrm{SU}(2^N)$, we define
    \begin{equation} \label{eqn:RAd}
      \overline{\mathcal{A}{\rm d}} (G) :=
            \{\mathcal{A}{\rm d}_x  \in {\rm Inn}(\mathrm{SU}(2^N)) \mid x \in G  \}.
    \end{equation}
    Here we use $\overline{\mathcal{A}{\rm d}} (G)$ instead of $\mathcal{A}{\rm d}(G)$
    to indicate that  elements of this group are adjoint actions on the ambient group $\mathrm{SU}(2^N)$, not
    adjoint actions on the subgroup $G$.
     For any positive integer $k \leq N$, define
     \begin{equation} \label{eqn:Wk}
     W_k := {\rm Span}_{\mathbb{R}} \{ \tilde{\sigma}_I \mid I \in \mathcal{I}_N^0, \,\,\, w(\tilde{\sigma}_I)=k\}
        \subset \mathfrak{su}(2^N),
     \end{equation}
     where $w(\tilde{\sigma}_I)$ is the weight of $\tilde{\sigma}_I$. We also define
     \[
    \mathrm{SU}(2)^{\otimes N}
        := \{U_1\otimes \dots \otimes U_N \mid U_i \in \mathrm{SU}(2), 1 \leq i \leq N\} \subset \mathrm{SU}(2^N).
    \]

    \begin{lemma} \label{lem:LieAlgSU2N}
     The Lie algebra of $\mathrm{SU}(2)^{\otimes N}$ is $W_1$.
    \end{lemma}

    \noindent
    {\bf Proof}:
    Let $H = \mathrm{SU}(2)^{\otimes N}$.
    Take any arbitrary smooth curve
    $\gamma(t) = U_{1}(t) \otimes \cdots \otimes U_N(t)$ in $H$ with $U_j(t) \in \mathrm{SU}(2)$ and $U_j(0) = \sigma_0$ for all $1 \leq j \leq N$. By the product rule for derivatives of tensors, we have
    \[
    \frac{\mathrm{d}}{\mathrm{d}t}\bigg|_{t = 0}\gamma(t)
        = \sum_{j=1}^N \sigma_0 \otimes \cdots \otimes \sigma_0 \otimes U_j'(0) \otimes \sigma_0 \otimes \cdots\otimes \sigma_0 \in W_1.
    \]
    Conversely, for any  $X \in W_1$, we can write
    $X = \sum_{j=1}^N X_j$, where $X_j=\sigma_0 \otimes \cdots \otimes v_j \otimes \cdots \otimes \sigma_0 $ with $v_j \in \mathfrak{su}(2)$. Since $X_1, \ldots , X_N$ commute with each other, we have
    \[
    \exp(tX) = \prod_{j=1}^N \exp(t X_j) = \exp(t v_1) \otimes \cdots \otimes \exp(t v_N) \in H.
     \]
    This completes the proof of the lemma.
    $\Box$

    \begin{remark}
    Although the definition of $\mathrm{SU}(2)^{\otimes N}$ is quite natural from the forms of its elements, the appearance of the tensor notation here might cause some confusions. In fact,
    let $\mathrm{SU}(2)^{N}$ be the product of $N$ copies of $\mathrm{SU}(2)$. Then
    Lemma~\ref{lem:LieAlgSU2N} implies that $\mathrm{SU}(2)^{\otimes N}$ and $\mathrm{SU}(2)^{N}$
    have isomorphic Lie algebras. We will see in the proof of Lemma \ref{lem:SU2SO3N} that
    $\mathrm{SU}(2)^{\otimes N}$ is isomorphic to a quotient of $\mathrm{SU}(2)^{N}$ by a finite group.
    In particular, dimension of $\mathrm{SU}(2)^{\otimes N}$ is $3N$, not $3^N$.
    \end{remark}

    The main result of this section is the following: 
    \begin{theorem}\label{isotropycomponent}
    Suppose $g = g_{\bf q}$ is a penalty metric on $\mathrm{SU}(2^N)$ satisfying the condition
    \begin{equation}\label{cond:V1_invariant}
    h_{*e}(W_1) = W_1 \quad {\rm for \,\,\, all \,\,\,} h \in \mathrm{Aut}(\mathrm{SU}(2^N), g).
    \end{equation}
    Then the identity component of the isometry group is
    \begin{equation}
    \mathrm{Isom}(\mathrm{SU}(2^N), g)_0
    = R(\mathrm{SU}(2^N)) \rtimes
                    \overline{\mathcal{A}{\rm d}}(\mathrm{SU}(2)^{\otimes N}),
    \end{equation}
    where $R(\mathrm{SU}(2^N))$ is the group of all right translations on $\mathrm{SU}(2^N)$.
    In particular, the identity component of the isotropy subgroup at $e$ is given by
    \begin{equation} \label{eqn:Isom-e-0}
    \mathrm{Isom}_e(\mathrm{SU}(2^N), g)_0 = \overline{\mathcal{A}{\rm d}}(\mathrm{SU}(2)^{\otimes N}).
    \end{equation}
    \end{theorem}

    Furthermore, we will show that Condition \eqref{cond:V1_invariant} is satisfied by the cliff, binomial, and exponential metrics, which are the most commonly used penalty metrics in geometric quantum  complexity.

    We first make some necessary preparations for the proof of Theorem \ref{isotropycomponent}.

    \begin{lemma}\label{1localproperty}
     For any $U \in \mathrm{SU}(2)^{\otimes N}$, $X \in W_1$, and $1 \leq k \leq N$, we have
     \[ \mathrm{Ad}_U (W_k) = W_k,  \hspace{20pt}  \mathrm{ad}_{X} (W_k) \subset W_k, \]
      where $\mathrm{ad}_X : \mathfrak{su}(2^N) \to \mathfrak{su}(2^N)$ is the adjoint action
      of the Lie algebra $\mathfrak{su}(2^N)$ on itself.
      In particular, $W_1$ is a Lie subalgebra of $\mathfrak{su}(2^N)$.
    \end{lemma}

    \noindent
    {\bf Proof}:
    Assume $U = U_1 \otimes \cdots \otimes U_N$ with $U_j \in \mathrm{SU}(2)$.
    For any $I = (i_1, \ldots, i_N) \in \mathcal{I}_N^0$,
    $\tilde\sigma_{I} = -i\sigma_{i_1} \otimes \cdots \otimes \sigma_{i_N}$ and
    \[
    \mathrm{Ad}_U(\tilde\sigma_{I}) = -i(U_1\sigma_{i_1}U_1^{-1}) \otimes \cdots \otimes (U_N \sigma_{i_N}U_N^{-1}).
    \]
    For any $1 \leq j \leq N$,
    $iU_j \sigma_{i_j}U_{j}^{-1} = {\rm Ad}_{U_j} (i \sigma_{i_j}) \in \mathfrak{su}(2)$ if $i_j \neq 0$, and $U_{j}\sigma_{i_j}U_j^{-1} = \sigma_0$ if $i_j = 0$. It follows that if $w(\tilde{\sigma}_I) = k$, then
    $\mathrm{Ad}_U(\tilde\sigma_I) \in W_k$.
    Therefore, $\mathrm{Ad}_U (W_k) \subset W_k$ for all $k$. Since $\mathrm{Ad}_U$ is invertible,
    we have $\mathrm{Ad}_U (W_k) = W_k$ for all $k$.

    For any $X \in W_1$, by Lemma \ref{lem:LieAlgSU2N}, $\exp(tX) \in \mathrm{SU}(2)^{\otimes N}$
    for all $t \in \mathbb{R}$.
    Since $\mathrm{ad}_X = \left. \frac{d}{dt} \right|_{t=0} \mathrm{Ad}_{\exp(tX)}$,
    we have $\mathrm{ad}_X (W_k) \subset W_k$ for all $k$.
    So the lemma is proved.
     $\Box$

    \begin{proposition}\label{prop:local_unitary_isometry}
    Let $g = g_{\bf q}$ be a penalty metric on $\mathrm{SU}(2^N)$. Then for any $U  \in \mathrm{SU}(2)^{\otimes N}$, its adjoint action $\mathrm{Ad}_U$ on $\mathfrak{su}(2^N)$ is a linear isometry with respect to $g$. Consequently, we have
    \begin{equation}
    \overline{\mathcal{A}{\rm d}}(\mathrm{SU}(2)^{\otimes N}) \subseteq \mathrm{Isom}_e(\mathrm{SU}(2^N), g).
    \end{equation}
    \end{proposition}

    \noindent
    {\bf Proof}:
    Let $g_1$ be the penalty metric on ${\rm SU}(2^N)$ with all penalty factors equal to 1.
    By Lemma~\ref{lem:AdInv}, $g_1$ is a bi-invariant metric. In particular,
    $\mathrm{Ad}_U^* g_1 = g_1$.

    For any penalty metric $g = g_{\bf q}$ with ${\bf q}=(q_1, \ldots, q_N)$,
     $g|_{W_k} = q_k \, g_1 |_{W_k}$ for all $k$. By Lemma \ref{1localproperty}, $\mathrm{Ad}_U (W_k) = W_k$.
    It follows that ${\rm Ad}_U |_{W_k}$ is a linear isometry with respect to $g$
    for all $k$. Since $W_k$ is orthogonal to $W_l$ with respect to $g$ if $k \neq l$,
    this implies that ${\rm Ad}_U$ is a linear isometry on $\mathfrak{su}(2^N)$ with respect to $g$.
    By Lemma \ref{automorphism_isometry}, $\mathcal{A}{\rm d}_U$ is an isometry on $\mathrm{SU}(2^N)$.
    The proposition is thus proved.
    $\Box$

    \begin{lemma} \label{lem:ad-non-0}
    For any $I \in \mathcal{I}_N^0$ and $X= \sum_{J \in \mathcal{I}_N^0} a_J \sigma_J$
    with $a_J \in \mathbb{C}$,
    if there exists $K \in \mathcal{I}_N^0$ such that $a_K [\sigma_I, \sigma_K] \neq 0$,
    then $[\sigma_I, X] \neq 0$.
    \end{lemma}

    \noindent
    {\bf Proof}:
    Let $\mathcal{J}_1 := \{ J \in \mathcal{I}_N^0 \mid [\sigma_I, \sigma_J] =0 \}$
      and $\mathcal{J}_2 := \{ J \in \mathcal{I}_N^0 \mid [\sigma_I, \sigma_J]  \neq 0 \}$.
     Let $X_1 := \sum_{J \in \mathcal{J}_1} a_J \sigma_J$ and $X_2 := \sum_{J \in \mathcal{J}_2} a_J \sigma_J$.
     Since $K \in \mathcal{J}_2$, $X_2 \neq 0$.
     By equation \eqref{eqn:IIJ=4J}, $[\sigma_I, [\sigma_I, X_2]] = 4 X_2 \neq 0$.
     Hence $[\sigma_I, X] = [\sigma_I, X_2] \neq 0$.
     The lemma is thus proved.
     $\Box$

    \begin{lemma}\label{preserve1local}
    For $X \in \mathfrak{su}(2^N)$, if $\mathrm{ad}_X(W_1) \subseteq W_1$, then
    $X \in W_1$.
    \end{lemma}

    \noindent
    {\bf Proof}:
    Write $X = \sum_{I \in \mathcal{I}_N^0} a_I\tilde\sigma_I = \sum_{k=1}^N X_k$ with $a_I \in \mathbb{R}$
    and $X_k \in W_k$.
   Suppose there exists $I_0$ with $k_0:=w(\tilde\sigma_{I_0}) > 1$ such that $a_{I_0} \neq 0$.  We can write $\tilde\sigma_{I_0} = -i\sigma_{i_1} \otimes \cdots \otimes \sigma_{i_N}$, where at least one component, say $\sigma_{i_u}$, is distinct from $\sigma_0$. We can then choose $\sigma_{j_u}$ with
   $j_u \in \{1,2,3\}$ such that $[\sigma_{i_u},\sigma_{j_u}] \neq 0$.
    This implies that if we choose
    \[ \tilde\sigma_{J_0} = -i\sigma_0 \otimes \cdots \otimes\sigma_0 \otimes \sigma_{j_u} \otimes\sigma_0\otimes \cdots \otimes \sigma_0 \in W_1,\]
    then
    $[\tilde\sigma_{I_0},\tilde\sigma_{J_0}] \neq 0. $
    By Lemma \ref{lem:ad-non-0}, $[X_{k_0}, \tilde\sigma_{J_0}] \neq 0$.
    By Lemma \ref{1localproperty}, $[X_k, \tilde\sigma_{J_0}] \in W_k$ for all $k$.
     Hence $[X, \tilde\sigma_{J_0}] \notin W_1$, which contradicts the assumption that
     $\mathrm{ad}_X(W_1) \subseteq W_1$.
     Therefore we must have $X \in W_1$. This completes the proof of the lemma.
     $\Box$

    We are now ready to prove Theorem \ref{isotropycomponent}.

    {\bf Proof of Theorem \ref{isotropycomponent}}:
    Since $\mathrm{SU}(2^N)$ is a compact connected simple Lie group,
    by Corollary \ref{cor:simple-component}, we only need to prove equation \eqref{eqn:Isom-e-0}.
    By Proposition \ref{simpleLiegp}, we have
    \begin{equation}
    \mathrm{Isom}_e(\mathrm{SU}(2^N), g)_0 \subseteq
                \mathrm{Inn}(\mathrm{SU}(2^N)) \cap \mathrm{Isom}(\mathrm{SU}(2^N), g).
    \end{equation}
    \noindent Therefore, to compute $\mathrm{Isom}_e(\mathrm{SU}(2^N), g)_0$, it suffices to consider
    inner automorphisms on $\mathrm{SU}(2^N)$ which preserve the metric $g$.
    Let $G= \{U \in \mathrm{SU}(2^N) \mid \mathcal{A}\mathrm{d}_U^* g = g\}$. Clearly $G$ is a closed subgroup of $\mathrm{SU}(2^N)$ and $\mathrm{Isom}_e(\mathrm{SU}(2^N), g)_0 = \overline{\mathcal{A}\mathrm{d}}(G_0)$,
     where $G_0$ is the identity component of $G$.
    	
    By condition \eqref{cond:V1_invariant}, $\mathrm{ad}_X(W_1) \subseteq W_1$ for all $X \in \mathfrak{g}$,
    where $\mathfrak{g}$ is the Lie algebra of $G$.
    So by Lemma \ref{preserve1local}, $\mathfrak{g} \subseteq W_1$.
    Consequently,  by Lemma \ref{lem:LieAlgSU2N}, $G_0 \subseteq \mathrm{SU}(2)^{\otimes N}$.
    Conversely, by Proposition~\ref{prop:local_unitary_isometry}, $\mathrm{SU}(2)^{\otimes N} \subseteq G_0$.
    Hence $G_0 = \mathrm{SU}(2)^{\otimes N}$. The theorem is thus proved.
     $\Box$

    Theorem \ref{isotropycomponent}  requires Condition \eqref{cond:V1_invariant} on the penalty metric. We want to show that this condition holds for common penalty metrics. We first establish the following proposition:

    \begin{proposition}\label{condition_isotropy}
    Let $g = g_{\bf q}$ be a penalty metric on $\mathrm{SU}(2^N)$ with penalty factors ${\bf q}=(q_1, \cdots,  q_N)$. If $q_1 \neq q_i$ for all $2 \leq i \leq N$, then $g$ satisfies Condition \eqref{cond:V1_invariant}. In particular, the binomial and exponential metrics (see Examples \ref{ex:Binomial} and \ref{ex:exponential}) satisfy Condition \eqref{cond:V1_invariant}.
    \end{proposition}

    \noindent
    {\bf Proof}:
    Let $g_1$ be the  penalty metric with all penalty factors equal to $1$.
    By the uniqueness of bi-invariant metrics on compact simple Lie groups (see
    \cite[Proposition 2.48]{AB}), up to a scaling factor, $g_1$ is the same as the bi-invariant metric
    defined by the inner product $\langle X, \, Y \rangle = \mathrm{tr}(X^\dagger Y)$ for all
    $X,Y \in \mathfrak{su}(2^N)$, where $X^\dagger$ denotes the complex conjugate of the transpose of $X$.
    Since inner automorphisms and the complex conjugation preserve $\langle \cdot, \, \cdot \rangle$, by
    equation \eqref{eqn:AutSU(n)}, we have
    $ \mathrm{Aut}(\mathrm{SU}(2^N)) \subseteq \mathrm{Isom}_e(\mathrm{SU}(2^N), g_1).$

    Let $\Psi: \mathfrak{su}(2^N) \to \mathfrak{su}(2^N)$ be the linear isomorphism defined by
    \begin{equation} \label{eqn:CalG}
    \Psi(X) := q_k X
    \end{equation}
    for all $X \in W_k$ and  $k=1, \ldots, N$.
    Then
    $g(X,Y) = g_1(X,\Psi(Y))$ for all $X,Y \in \mathfrak{su}(2^N)$.
    Therefore, for any $h \in \mathrm{Aut}(\mathrm{SU}(2^N), g)$, we have
    \[
    \begin{aligned}
    g_1(h_{*e}X,\Psi(h_{*e}Y))  &= g(h_{*e}X,h_{*e}Y)
    = g(X,Y) = g_1(X,\Psi(Y))
    = g_1(h_{*e}X,h_{*e}(\Psi(Y))),
    \end{aligned}
    \]
    for all $X,Y \in \mathfrak{su}(2^N)$. Since $g_1$ is non-degenerate, this implies
    \[
    h_{*e} \circ \Psi = \Psi \circ h_{*e}.
    \]
    \noindent Thus $h_{*e}$ preserves the eigenspaces of $\Psi$.

    If $q_1 \neq q_i$ for all $2 \leq i \leq N$,
    the eigenspace of $\Psi$ with eigenvalue $q_1$ is precisely $W_1$. Consequently,
    $h_{*e}(W_1) = W_1$, i.e. condition \eqref{cond:V1_invariant} is satisfied.

    Now it only remains to verify that $q_1 \neq q_k$ for all $2 \leq k \leq N$ for the exponential
    and  binomial metrics.

    For the exponential metric, the penalty factors are given by $q_k = x^{2k}$, where $x > 1$ is a constant. Since $x^{2k}$ is strictly increasing as $k$ increases, it follows immediately that $q_1 \neq q_k$ for all $2 \leq k \leq N$.

    For the binomial metric, the penalty factors are given by $q_k = \left(\binom{N}{k}3^k\right)^{\alpha}$ for $k = 1, \dots, N$, where $\alpha > 0$ is a constant. Since the function $f(x) = x^\alpha$ is strictly increasing for
    $x > 0$, it suffices to check the case $\alpha=1$.
    In this case, $q_N = 3^N > 3N$ if $N >1$. For $2 \leq k < N$,
    $q_k = \binom{N}{k}3^k > 3N$  since  $\binom{N}{k} \geq N$ and $3^k > 3$.
    Therefore, $ q_k > 3N = q_1$ for all $2 \leq k \leq N$.
    The proposition is thus proved.
     $\Box$

    Note that the cliff metric does not satisfy the condition in Proposition \ref{condition_isotropy}. To prove
    that the cliff metric also satisfies condition \eqref{cond:V1_invariant}, we need the following lemma: 

    \begin{lemma} \label{lem:W2W2-W3}
    Let $\pi_k: \mathfrak{su}(2^N) \to W_k$ be the natural projection associated with the decomposition $\mathfrak{su}(2^N) = \bigoplus_{k=1}^N W_k$. Assume $N \geq 3$.
    For any $X \in \mathfrak{su}(2^N)$ with $\pi_2(X) \neq 0$, there exists $Y \in W_2$ such that $\pi_3([X,Y]) \neq 0$.
    \end{lemma}

    \noindent
    {\bf Proof}:
    Write $X = \sum_{I \in \mathcal{I}_N^0} a_I \tilde\sigma_I$ with $a_I \in \mathbb{R}$. Since
    $\pi_2(X) \neq 0$, there exists $I_0 = (i_1, \ldots, i_N) \in \mathcal{I}_N^0$ such that $a_{I_0} \neq 0$ and $w(\tilde\sigma_{I_0}) = 2$. There exist $1 \leq u < v \leq N$ such that $i_k \neq 0$ if and only if
    $k \in \{u, v\}$.
    We can choose $J=(j_1, \ldots, j_N) \in \mathcal{I}_N^0$ such that $j_u \notin \{0, i_u\}$, $j_t \neq 0$
    for some $t \notin \{u, v\}$, and $j_l=0$ for all $l \notin \{u, t\}$.
    Then by Lemma \ref{Paulimainthm}, $w(\tilde\sigma_J)=2$ and
    \begin{equation} \label{eqn:JI0-W3}
    0 \neq [\tilde\sigma_J, \tilde\sigma_{I_0}] \in W_3.
    \end{equation}

    Let $\mathcal{J}_1 := \{I \in \mathcal{I}_N^0 \mid [\tilde\sigma_J, \tilde\sigma_{I}] = 0 \}$
    and $\mathcal{J}_2 := \{I \in \mathcal{I}_N^0 \mid [\tilde\sigma_J, \tilde\sigma_{I}] \neq 0 \}$.
    By equation~\eqref{eqn:IIJ=4J}, vectors $\{ [\tilde\sigma_J, \tilde\sigma_{I}] \mid I \in \mathcal{J}_2 \}$
    are linearly independent.
    Since $[\tilde\sigma_J, X] = \sum_{I \in \mathcal{J}_2} a_I [\tilde\sigma_J, \tilde\sigma_{I}]$,
    equation \eqref{eqn:JI0-W3} implies that $\pi_3([\tilde\sigma_J, X]) \neq 0$.
    Consequently, taking $Y = \tilde\sigma_{J} \in W_2$, we have $\pi_3([X,Y]) \neq 0$.
    The lemma is thus proved.
    $\Box$

    Note that the cliff metric is bi-invariant if $N=2$.

    \begin{proposition}\label{cliff_metric_preserve_locality}
    Assume $N \geq 3$. Let $g_q$ ($q > 1$) be the cliff metric on $\mathrm{SU}(2^N)$ defined in Example \ref{ex:Cliff}. Then $g_q$ satisfies Condition \eqref{cond:V1_invariant}.
    \end{proposition}

    \noindent
    {\bf Proof}:
        As in the proof of Proposition \ref{condition_isotropy}, for any $h \in \mathrm{Aut}(\mathrm{SU}(2^N), g_q)$,
    its differential $h_{*e}$  preserves the eigenspaces of the linear isomorphism $\Psi$
    defined by equation \eqref{eqn:CalG}.
    Let
    \[ V_1 := W_1 \oplus W_2, \hspace{20pt} V_2 := \bigoplus_{k=3}^N W_k.\]
    Then $V_1$ and $V_2$ are the eigenspaces of $\Psi$ with eigenvalues $1$ and $q$ respectively. Hence we have
    \[
    h_{*e}(V_1) = V_1, \quad \text{and} \quad h_{*e}(V_2) = V_2.
    \]
    Since $h_{*e}$ is an isomorphism, to show $h_{*e}(W_1) = W_1$, it suffices to prove
    $h_{*e}(W_1) \subset W_1$.

    Suppose, for the sake of contradiction, that $h_{*e}(W_1) \not \subseteq W_1$. Then there exists  $X \in W_1$ such that $h_{*e}(X) \in V_1$ but $\pi_2(h_{*e}(X) ) \neq 0$,  where $\pi_k$ is the map defined in Lemma \ref{lem:W2W2-W3}. By Lemma \ref{lem:W2W2-W3}, there exists  $Y \in W_2 \subset V_1$ such that $\pi_3([h_{*e}(X), Y]) \neq 0$.

    Since $h_{*e}(V_1) = V_1$, there exists $Z \in V_1$ such that $h_{*e}(Z) = Y$.
     Since $X \in W_1$, Lemma \ref{1localproperty} guarantees that $[X,Z] \in V_1$. Since $h_{*e}$ preserves $V_1$, it follows that $h_{*e}([X,Z]) \in V_1$.
    Because $h$ is a Lie group automorphism, its differential $h_{*e}$ is a Lie algebra isomorphism. Hence we have
    \[
    [h_{*e}(X),Y] = [h_{*e}(X),h_{*e}(Z)] = h_{*e}([X,Z]) \in V_1,
    \]
    which contradicts the fact that $\pi_3([h_{*e}(X), Y]) \neq 0$.
    This completes the proof of the proposition.
     $\Box$

\section{Normalizer of $\mathrm{Isom}_e(\mathrm{SU}(2^N), g)_0$ in
            $\mathrm{O}(\mathfrak{su}(2^N), g)$}

     For a penalty metric $g$ on $\mathrm{SU}(2^N)$ satisfies Condition~\eqref{cond:V1_invariant}, we have computed the identity component of
    $\mathrm{Isom}_e(\mathrm{SU}(2^N), g)$ in Theorem \ref{isotropycomponent}.
    To compute $\mathrm{Isom}_e(\mathrm{SU}(2^N), g)$, we first embed this group into  $\mathrm{O}(\mathfrak{su}(2^N), g)$, which
    is the group of all linear transformations on  $\mathfrak{su}(2^N) = T_e \mathrm{SU}(2^N)$ preserving the metric
      induced by $g$. First, there is a natural group homomorphism
      \[ \mathcal{L}: \mathrm{Isom}_e(\mathrm{SU}(2^N), g) \longrightarrow \mathrm{O}(\mathfrak{su}(2^N), g) \]
    defined by $\mathcal{L}(\varphi) = \varphi_{*e}$, which is the differential of $\varphi$ at the identity element $e$.
    By a basic property of isometries (see \cite[Lemma 11.2 in Chapter I]{Helgason}),  for any $\varphi, \psi \in \mathrm{Isom}_e(\mathrm{SU}(2^N), g)$, $\varphi = \psi$ if and only if
    $\varphi_{*e} = \psi_{*e}$. Hence $\mathcal{L}$ is injective and
    $\mathrm{Isom}_e(\mathrm{SU}(2^N), g)$ is isomorphic to the subgroup
    $\mathcal{L}(\mathrm{Isom}_e(\mathrm{SU}(2^N), g)) \subset \mathrm{O}(\mathfrak{su}(2^N), g)$.
    By Theorem \ref{isotropycomponent},
    \[
    \mathcal{L}(\mathrm{Isom}_e(\mathrm{SU}(2^N), g)_0)
    = \mathcal{L}(\overline{\mathcal{A}\mathrm{d}}(\mathrm{SU}(2)^{\otimes N}))
    = \overline{\mathrm{Ad}}(\mathrm{SU}(2)^{\otimes N}),
    \]
    where
    \[ \overline{\mathrm{Ad}}(\mathrm{SU}(2)^{\otimes N})
        := \{{\rm Ad}_x \in {\rm Inn}(\mathfrak{su}(2^N)) \mid x \in \mathrm{SU}(2)^{\otimes N}\}. \]
    Let 
    \begin{equation} \label{eqn:CalGH}
     \mathcal{G} := \mathrm{O}(\mathfrak{su}(2^N), g), \hspace{20pt}
    \mathcal{H} := \overline{\mathrm{Ad}}(\mathrm{SU}(2)^{\otimes N}).
    \end{equation}
    Note that $\mathcal{H}$ is the identity component
    of $\mathcal{L}(\mathrm{Isom}_e(\mathrm{SU}(2^N), g))$, which
    is a subgroup of $\mathcal{G}$. We now construct some possible candidates
    for other elements in $\mathcal{L}(\mathrm{Isom}_e(\mathrm{SU}(2^N), g))$.

   Let $S_N$ be the group of all permutations of the set $\{1, \ldots, N\}$. For any $\tau \in S_N$,
   let $\psi_\tau: \mathfrak{su}(2^N) \to \mathfrak{su}(2^N)$ be the linear map defined by
   \begin{equation} \label{eqn:psitau}
    \psi_\tau ( i \, \sigma_{j_1} \otimes \cdots \otimes \sigma_{j_N})
            := i \, \sigma_{j_{\tau(1)}} \otimes \cdots \otimes \sigma_{j_{\tau(N)}}
    \end{equation}
   for all $(j_1, \ldots, j_N) \in \mathcal{I}_N^0$.
   Let
   \begin{equation} \label{eqn:Sbar}
    \overline{S}_N := \{ \psi_\tau \mid \tau \in S_N\}.
    \end{equation}
    Then $\overline{S}_N$ is a group
   isomorphic to $S_N$.
   Let
   \begin{equation} \label{eqn:CalP}
    \mathcal{P} := \bigcup_{k=1}^N \{(i_1, \ldots, i_k) \mid 1 \leq i_1 < i_2 < \cdots < i_k \leq N\}.
    \end{equation}
   For any $I=(i_1, \ldots, i_k) \in \mathcal{P}$, define
    \begin{equation}\label{def_irr}
    V_I := \mathrm{Span}_{\mathbb{R}}\{-i\sigma_{j_1 }\otimes \cdots \otimes \sigma_{j_N}
          \mid j_t \neq 0 \,\,\, \text{\rm if and only if} \,\,\, t \in \{i_1, \cdots, i_k\} \}.
    \end{equation}
    Then $\mathfrak{su}(2^N) = \bigoplus_{I \in \mathcal{P}} V_I$.
    For any map $\epsilon: \mathcal{P} \to \{1, -1\}$, let
    $\phi_\epsilon: \mathfrak{su}(2^N) \to \mathfrak{su}(2^N)$ be the linear map defined by
    \begin{equation} \label{eqn:phiepsilon}
     \phi_\epsilon \left( \sum_{I \in \mathcal{P}} X_I \right) :=  \sum_{I \in \mathcal{P}} \epsilon(I) X_I
     \end{equation}
    for all $X_I \in V_I$. Let
    \begin{equation} \label{eqn:Phi}
     \Phi := \{ \phi_\epsilon \mid \epsilon: \mathcal{P} \to \{1, -1\} \}.
     \end{equation}
    Then $\Phi$ is a group isomorphic to $\mathbb{Z}_2^{2^N - 1}$.

    For any $\tau \in S_N$, $\epsilon: \mathcal{P} \to \{1, -1\}$, $J \in \mathcal{I}_N^0$,
    both $\psi_\tau(\tilde\sigma_J)$ and $\phi_\epsilon(\tilde\sigma_J)$ have the form $\pm \tilde\sigma_K$
    for some $K \in \mathcal{I}_N^0$ with $w(\tilde\sigma_K) = w(\tilde\sigma_J)$.
    Hence $\psi_\tau, \phi_\epsilon \in \mathcal{G}$. So we have
    \begin{equation}  \label{eqn:SPhiInG}
    \overline{S}_N \subset \mathcal{G}, \hspace{20pt}
    \Phi \subset \mathcal{G}.
    \end{equation}

    The main purpose of this section is to prove the following theorem: 

    \begin{theorem}\label{normalizer}
    Let $g$ be a penalty metric on $\mathrm{SU}(2^N)$ satisfying Condition \eqref{cond:V1_invariant}. Then
    \begin{equation}
      \mathcal{L}(\mathrm{Isom}_e(\mathrm{SU}(2^N), g))  \subseteq
        \left( \,\mathcal{H} \times \Phi \, \right) \rtimes \overline{S}_N
      \cong \left( \mathrm{SO}(3)^N \times \mathbb{Z}_2^{2^N - 1} \right) \rtimes S_N,
    \end{equation}
    where $\mathcal{H}$, $\overline{S}_N$ and $\Phi$ are defined by equations \eqref{eqn:CalGH}, \eqref{eqn:Sbar} and \eqref{eqn:Phi} respectively.
    \end{theorem}

    To prove this theorem,  we need the following lemma:

    \begin{lemma}\label{normalizer_basic_lemma}
    Let $G$ be a  Lie group and $H \subseteq G$ be a closed connected subgroup. If $W \subseteq G$ is a subgroup with identity component $W_0 = H$, then $W \subseteq {\rm N}_G(H)$, where
    \[ {\rm N}_G(H):=\{ x \in G \mid x H x^{-1} \subset H\}\]
     is the normalizer of $H$ in $G$.
    \end{lemma}

    \noindent
    {\bf Proof}:
    For any $x \in W$, we want to show that $x H x^{-1} \subseteq H$. Since the map
    $H \rightarrow W$ defined by $h \mapsto x h x^{-1}$ is continuous and $H$ is connected, $xHx^{-1}$ is a connected subset of $W$ which contains the identity element $e$. Hence $xHx^{-1} \subseteq W_0 =  H$.
     $\Box$

    Applying this lemma to the special case where  
    $G$ and $H$ are  respectively $\mathcal{G}$ and  $\mathcal{H}$ defined by equation \eqref{eqn:CalGH},
    and $W=\mathcal{L}(\mathrm{Isom}_e(\mathrm{SU}(2^N), g))$,
    we have
    \begin{equation} \label{eqn:IsomNormal}
    \mathcal{L}(\mathrm{Isom}_e(\mathrm{SU}(2^N), g)) \subseteq {\rm N}_\mathcal{G} (\mathcal{H}).
     \end{equation}
    In the rest part of this section, we will compute the normalizer ${\rm N}_\mathcal{G} (\mathcal{H})$.

    Let $\theta: {\rm N}_\mathcal{G} (\mathcal{H}) \longrightarrow \mathrm{Aut}(\mathcal{H})$ be the group homomorphism defined by $\theta(x)(h) := x h x^{-1}$ for $x \in  {\rm N}_\mathcal{G} (\mathcal{H})$ and $h \in \mathcal{H}$.
    The kernel of $\theta$ is given by the centralizer of $\mathcal{H}$ in $\mathcal{G}$, i.e.
    \[   {\rm C}_\mathcal{G} (\mathcal{H}) :=
            \{ x \in  \mathcal{G} \mid x h x^{-1} = h \,\,\, {\rm for \,\,\, all \,\,\,} h \in \mathcal{H}\}.\]
    Let $\pi: \mathrm{Aut}(\mathcal{H}) \longrightarrow \mathrm{Out}(\mathcal{H})
        :=  \mathrm{Aut}(\mathcal{H})/\mathrm{Inn}(\mathcal{H})$ be the  natural projection,
        where $\mathrm{Out}(\mathcal{H})$ is the outer automorphism  group of $\mathcal{H}$.
    The kernel of $\pi \circ \theta:  {\rm N}_\mathcal{G} (\mathcal{H}) \longrightarrow \mathrm{Out}(\mathcal{H})$ is given by ${\rm C}_\mathcal{G} (\mathcal{H}) \cdot \mathcal{H}$, which is a normal subgroup of
    ${\rm N}_\mathcal{G} (\mathcal{H})$.
    Hence we have a short exact sequence of groups:

    \begin{equation}\label{shortexactsequence}
    1 \to {\rm C}_\mathcal{G} (\mathcal{H}) \cdot \mathcal{H} \to {\rm N}_\mathcal{G} (\mathcal{H}) \to
        {\rm N}_\mathcal{G} (\mathcal{H})/({\rm C}_\mathcal{G} (\mathcal{H}) \cdot \mathcal{H}) \to 1.
    \end{equation}

    \noindent Note that  ${\rm N}_\mathcal{G} (\mathcal{H})/({\rm C}_\mathcal{G} (\mathcal{H}) \cdot \mathcal{H})$ is isomorphic to
     $\pi \circ \theta ({\rm N}_\mathcal{G} (\mathcal{H}))$, which is a subgroup of $\mathrm{Out}(\mathcal{H})$.
    The computation of ${\rm N}_\mathcal{G} (\mathcal{H})$ will be divided into the following three steps:
    \begin{enumerate}
    \item[(i)] In Section 4.1, we will calculate $\mathrm{Out}(\mathcal{H})$.

    \item[(ii)] In Section 4.2, we calculate the centralizer ${\rm C}_\mathcal{G} (\mathcal{H})$ using representation theory.

    \item[(iii)] In Section 4.3, we prove 
    ${\rm N}_\mathcal{G} (\mathcal{H})/({\rm C}_\mathcal{G} (\mathcal{H}) \cdot \mathcal{H})$ is isomorphic to
      $\mathrm{Out}(\mathcal{H})$ and verify that the short exact sequence \eqref{shortexactsequence} splits. 
    \end{enumerate}

The proof of Theorem \ref{normalizer} will also be given in Section 4.3.

\subsection{Computing $\mathrm{Out}(\mathcal{H})$}

    We first give a more precise group structure for
    $\mathcal{H} = \overline{\mathrm{Ad}}(\mathrm{SU}(2)^{\otimes N})$.
        \begin{lemma} \label{lem:SU2SO3N}
        $\mathcal{H}  \cong \mathrm{SO}(3)^N$.
        \end{lemma}

        \noindent
        {\bf Proof}:
         First, we have a Lie group homomorphism
         $\rho: \mathrm{SU}(2)^{N}  \longrightarrow \mathrm{SU}(2)^{\otimes N}$ defined by
        \[
        \begin{aligned}
            \rho(U_1,\cdots,U_N) & = U_1 \otimes \cdots \otimes U_N.
        \end{aligned}
        \]
        The kernel of this map is given by
        $\ker(\rho) = \left\{ (a_1 \sigma_0, a_2 \sigma_0, \cdots, a_N \sigma_0) \left| a_i = \pm 1, \prod_{i = 1}^N a_i = 1 \right\} \right.$.

        Note that the kernel of the homomorphism
        $\mathrm{Ad}: \mathrm{SU}(2^N) \longrightarrow {\rm Inn}(\mathfrak{su}(2^N))$
        is $Z(\mathrm{SU}(2^N))$, which is the center of $\mathrm{SU}(2^N)$.
        Since $\mathrm{SU}(2)^{\otimes N} \cap Z(\mathrm{SU}(2^N)) = \{\pm \mathrm{id}\}$,
        we have
        \[ \ker(\mathrm{Ad} \circ \rho) = \{(a_1\sigma_0, a_2\sigma_0, \cdots, a_N\sigma_0) \mid a_i = \pm 1 \}
           \cong \mathbb{Z}_2^N. \]
        Consequently, $\overline{\mathrm{Ad}}(\mathrm{SU}(2)^{\otimes N})= {\rm Im}(\mathrm{Ad} \circ \rho)
           \cong (\mathrm{SU}(2)/\mathbb{Z}_2)^{N} \cong \mathrm{SO}(3)^N$.
         $\Box$

        A nontrivial group $G$ is called a {\it simple group} if it has no normal subgroups other than $\{e\}$ and $G$
        (see \cite[Page 16]{robinson1996}).
        Note that a Lie group is simple if its Lie algebra is simple (see \cite[Page 131]{Helgason}). So a simple Lie group may not be a simple group. For example, ${\rm SU}(2)$ is a simple Lie group but it is not a simple group
        since it has a non-trivial center. The following result can be found in \cite[3.3.16]{robinson1996}:

        \begin{lemma} \label{lem:NSProd}
        Let $G_1, \ldots, G_n$ be non-abelian simple groups. Then any normal subgroup of $G_1 \times \cdots \times G_n$
        must be of the form
        $\{e\} \times \cdots \times G_{i_1} \times \cdots \times G_{i_k} \times \cdots \times \{e\}$ for some
        $1 \leq i_1 < \cdots < i_k \leq n$.
        \end{lemma}
        Using this lemma, we can prove the following:

        \begin{lemma}\label{simple_group_product}
        Let $G$ be a non-abelian simple group and $G^n$ the product of $n$ copies of $G$.
        Then $\mathrm{Aut}(G^n) \cong (\mathrm{Aut}(G))^n \rtimes S_n$.
        \end{lemma}

        \noindent
        {\bf Proof}:
        Let $G_i := \{e\} \times \cdots \times G \times \cdots \times \{e\}$ where $G$ occurs at the i-th position.
        By Lemma \ref{lem:NSProd}, $G_1, \ldots, G_n$ are the only minimal normal subgroups of $G^n$, i.e they do not contain other non-trivial normal subgroups.
        For any $\varphi \in \mathrm{Aut}(G^n)$, since $\varphi$ maps minimal normal subgroups to minimal normal subgroups, it permutes the set $\{G_1, \ldots, G_n\}$.
        Note that the group $S_n$ acts as automorphisms on $G^n$ by permuting factors in the product.
        After composing with the action by an element in $S_n$, we can assume that $\varphi(G_i) = G_i$
        for all $i=1, \ldots, n$. Let $\varphi_i$ be the restriction of $\varphi$ to $G_i$. Then
        $\varphi_i$ can be considered as an element in ${\rm Aut}(G)$ since $G_i \cong G$.
        It follows that $\varphi = \varphi_1 \times \cdots \times \varphi_n \in (\mathrm{Aut}(G))^n$.
        This completes the proof of the lemma.
        $\Box$

         The automorphism group of $\mathrm{SO}(3)$ might be well known. We include a proof here for completeness.

        \begin{lemma}\label{SO_3_property}
        $\mathrm{Aut}(\mathrm{SO}(3)) =\mathrm{Inn}(\mathrm{SO}(3)) \cong \mathrm{SO}(3)$.
        \end{lemma}

        \noindent{\bf Proof}:
        Since $\mathrm{SU}(2)$ is simply connected, $\mathrm{Aut}(\mathfrak{su}(2))
        \cong \mathrm{Aut}(\mathrm{SU}(2)) \cong \mathrm{SO}(3)$
        by equation \eqref{eqn:AutSU(n)}.
        Since $\mathfrak{so}(3) \cong \mathfrak{su}(2)$, we have
        $\mathrm{Aut}(\mathfrak{so}(3)) \cong \mathrm{SO}(3) \cong \mathrm{Inn}(\mathrm{SO}(3)).$
        Since the map
        $\mathrm{Aut}(\mathrm{SO}(3)) \rightarrow \mathrm{Aut}(\mathfrak{so}(3))$ is injective,
        we have $\mathrm{Aut}(\mathrm{SO}(3)) = \mathrm{Inn}(\mathrm{SO}(3))  \cong \mathrm{SO}(3)$. $\Box$

        Since $\mathrm{SO}(3)$ is a non-abelian simple group (see, for instance, Chapter 2, Section 3 of \cite{stillwell2008naive}), Lemma \ref{simple_group_product} and Lemma \ref{SO_3_property} immediately imply the following:

        \begin{lemma} \label{lem:AutOutSO3N}
        $\mathrm{Aut}(\mathrm{SO}(3)^N) \cong \mathrm{SO}(3)^N \rtimes S_N$ and
         $\mathrm{Out}(\mathrm{SO}(3)^N) \cong S_N$.
        \end{lemma}
    By Lemma \ref{lem:SU2SO3N}, this also gives $\mathrm{Aut}(\mathcal{H})$ and  $\mathrm{Out}(\mathcal{H})$.

\subsection{Computing centralizer ${\rm C}_\mathcal{G} (\mathcal{H})$}

        This subsection is devoted to the proof of the following
        \begin{lemma}\label{centralizer}
         ${\rm C}_\mathcal{G} (\mathcal{H}) = \Phi,$ where $\Phi$ is defined by equation \eqref{eqn:Phi}.
        \end{lemma}

        To prove Lemma \ref{centralizer}, we first recall some basic results from representation theory.
        Let $G$ be a group and $V$ be a finite-dimensional real vector space. The complexification of a real representation $\rho: G \to \mathrm{GL}(V)$ is a complex representation
        $\rho_{\mathbb{C}}: G \to \mathrm{GL}(V \otimes_{\mathbb{R}}\mathbb{C})$
         defined by $\rho_{\mathbb{C}}(h)(v + iw) = \rho(h)(v) + i\rho(h)(w)$
         for all $h \in G$ and $v, w \in V$. If $\rho_{\mathbb{C}}$ is an irreducible representation, then $\rho$ is called an {\it absolutely irreducible representation} (also called an irreducible $\mathbb{R}$-representation of real type in \cite[Chapter II Section 6]{BtD}).
        Let \[ \mathrm{End}_G(V) :=\{\varphi \in \mathrm{End}(V) \mid  \varphi \circ  \rho(h)
                                        = \rho(h) \circ \varphi, \forall h \in G\}. \]
        If $\rho: G \to \mathrm{GL}(V)$ is an absolutely irreducible real representation, the following result can be proved using Schur's lemma (see \cite[Chapter II Theorem 6.7]{BtD}):
        \[
        \mathrm{End}_G(V) = \mathbb{R} \, \, \mathrm{id}_V.
        \]

        Given representations $\rho_1: G_1 \to \mathrm{GL}(V_1)$ and $\rho_2: G_2 \to \mathrm{GL}(V_2)$ of two groups $G_1$ and $G_2$, their tensor product $\rho_1 \otimes \rho_2: G_1 \times G_2 \to \mathrm{GL}(V_1 \otimes V_2)$ is defined by
        \[
         (\rho_1 \otimes \rho_2)(h_1,h_2)(v_1 \otimes v_2) := \rho_1(h_1)v_1 \otimes \rho_2(h_2)v_2
        \]
        for all $h_i \in G_i$, $v_i \in V_i$. The following fact is also well known
        (see \cite[Proposition 4.14 on page 82 and Exercise 3 on page 264]{BtD}):
        \begin{lemma} \label{tensor}
        If $\rho_1$ and $\rho_2$ are absolutely irreducible real representations of $G_1$ and $G_2$ respectively,
        then $\rho_1 \otimes \rho_2$ is also absolutely irreducible as a representation of $G_1 \times G_2$ .
       \end{lemma}

        Let $\rho_0$  be the trivial representation of $\mathrm{SO}(3)$ on
        ${\rm Span}_{\mathbb{R}} \{  \sigma_0 \}$, and $\rho_1$ the representation of $\mathrm{SO}(3)$ on 
        ${\rm Span}_{\mathbb{R}} \{ \sigma_1,  \sigma_2,  \sigma_3 \}$ 
        induced by the adjoint action on 
        $\mathfrak{so}(3) \cong \mathfrak{su}(2) = {\rm Span}_{\mathbb{R}} \{ i \sigma_1, i \sigma_2, i \sigma_3 \}$.
         Note that $\rho_1$ is isomorphic to the standard representation of $\mathrm{SO}(3)$
       on $\mathbb{R}^3$ induced by matrix product 
        (see, for example, \cite[Page 75]{woit2017quantum}).
       In this section, we will consider a class of absolutely irreducible representations of $\mathrm{SO}(3)^N$
       which can be defined as tensor products of $\rho_0$ and $\rho_1$.

        \begin{lemma} \label{lem:V-I}
        For any $I=(i_1, \ldots, i_k) \in \mathcal{P}$, let $V_I$ be the real vector space defined by equation \eqref{def_irr}.  Let $\rho^I$  be the representation of $\mathrm{SO}(3)^N$ on $V_I$ defined by
        \[
        \rho^I:= \rho_1 \otimes \cdots \otimes \rho_N,
        \]
        \noindent where $\rho_j = \rho_1$ if $j \in \{i_1, \cdots, i_k\}$ and $\rho_j=\rho_0$ otherwise.
        Then $\rho^I$ is absolutely irreducible for all $I \in \mathcal{P}$. Moreover, $\rho^I$ is not isomorphic to $\rho^J$ for all $I, J \in \mathcal{P}$ with $I \neq J$.
        \end{lemma}

        \noindent
        {\bf Proof}:
        It is well known that $\rho_0$ and $\rho_1$ are absolutely irreducible (see, for example, \cite[Page 273]{BtD}).
         By Lemma \ref{tensor},  $\rho^I$ is absolutely irreducible for all $I \in \mathcal{P}$.
         Note that for all $U \in \mathrm{SU}(2)^{\otimes N}$, the action of
         $\rho^I(\mathrm{Ad}_U)$  is just the restriction of $\mathrm{Ad}_U$ on $V_I$
         when we consider $\mathrm{Ad}_U$ as an element in  $\mathrm{SO}(3)^N$  by Lemma \ref{lem:SU2SO3N}. 

        Assume $I=(i_1, \ldots, i_k), J=(j_1, \ldots, j_l) \in \mathcal{P}$ and $I \neq J$.
        At least one of the sets of indices, say $\{i_1, \ldots, i_k\}$, is not equal to $\{1, \ldots, N\}$.
        We can choose
        $t \in \{j_1, \ldots, j_l\} \setminus \{i_1, \ldots, i_k\}$ and
        $h=(h_1, \ldots, h_N) \in \mathrm{SO}(3)^N$ with  $h_s=e$ for $s \neq t$ and $h_t \neq e$, where $e$
        is the identity element of $\mathrm{SO}(3)$.
        Then $\rho^I (h)= {\rm id}_{V_I}$ and $\rho^J (h) \neq {\rm id}_{V_J}$.
        Hence $\rho^I$ is not isomorphic to $\rho^J$.        
        $\Box$

        Now we are ready to prove Lemma \ref{centralizer}.

        \noindent
        {\bf Proof of Lemma \ref{centralizer}}:

        Recall $\mathcal{G} = \mathrm{O}(\mathfrak{su}(2^N), g)$ and 
        $\mathcal{G} \supset \mathcal{H} = \overline{\mathrm{Ad}}(\mathrm{SU}(2)^{\otimes N}) \cong \mathrm{SO}(3)^N $.
        We can consider $\mathcal{G}$ as a subset of $\mathrm{End}(\mathfrak{su}(2^N))$. Then $\mathfrak{su}(2^N)$ is an $\mathcal{H}$-module which admits the following irreducible decomposition:
        \begin{equation} \label{eqn:sumVI}
        \mathfrak{su}(2^N) = \bigoplus_{I \in \mathcal{P}} V_I,
        \end{equation}
        where $\mathcal{P}$ and $V_I$ are defined by equations \eqref{eqn:CalP} and \eqref{def_irr}. The action of
        $\mathcal{H}$ on $V_I$ is given by $\rho^I$ defined in Lemma \ref{lem:V-I}.
        Moreover
        \[ {\rm C}_\mathcal{G} (\mathcal{H}) = \mathrm{End}_{\mathcal{H}} (\mathfrak{su}(2^N)) \cap \mathcal{G}.\]

        Since each $V_I$ is an absolutely irreducible  $\mathcal{H}$-module and $V_I$ is not isomorphic to $V_J$ if $I \neq J$,
        by Schur's lemma for absolutely irreducible representations, we have
        \[
        \begin{aligned}
        \mathrm{End}_{\mathcal{H}} (\mathfrak{su}(2^N)) 
            &=\bigoplus_{I \in \mathcal{P}} \mathrm{End}_{\mathcal{H}} (V_I)
            = \bigoplus_{I \in \mathcal{P}} \mathbb{R} \, \mathrm{id}|_{V_I}.
        \end{aligned}
        \]
        Let $\pi_I$ be the projection from $\mathfrak{su}(2^N)$ onto $V_I$
        with respect to the decomposition \eqref{eqn:sumVI}.
       Then for any $\varphi \in \mathrm{End}_{\mathcal{H}} (\mathfrak{su}(2^N))$, there exist
       $\lambda_I \in \mathbb{R}$ for $I \in \mathcal{P}$ such that
        \begin{equation} \label{eqn:Endo-H}
        \varphi(X) = \sum_{I \in \mathcal{P}} \lambda_I \, \pi_I(X)
        \end{equation}
        for all $X \in \mathfrak{su}(2^N)$.
        If $\varphi \in {\rm C}_\mathcal{G} (\mathcal{H})$, then $\varphi$ must preserve the metric $g$.
        For every $I \in \mathcal{P}$, choose
        $0 \neq X_I \in V_I$,  we have
        $
        g(X_I,X_I) = g(\varphi(X_I), \varphi(X_I)) = \lambda_I^2 \, g(X_I,X_I).
        $
       Hence $\lambda_I = \pm 1$ for all $I$ and $\varphi \in \Phi$, where $\Phi$ is defined by equation \eqref{eqn:Phi}.
        
        Conversely, if $\lambda_I = \pm 1$ for all $I \in \mathcal{P}$, it is straightforward to check the map $\varphi$ defined by equation \eqref{eqn:Endo-H} belongs to ${\rm C}_\mathcal{G} (\mathcal{H})$. Hence we have ${\rm C}_\mathcal{G} (\mathcal{H})=\Phi$.
         $\Box$

\subsection{Computing Normalizer ${\rm N}_\mathcal{G} (\mathcal{H})$ }

       In Lemma \ref{lem:AutOutSO3N}, we have shown $\mathrm{Out}(\mathcal{H}) \cong S_N$. We first see how to realize this group using the adjoint action of certain elements in $\mathrm{SU}(2^N)$.
      
        For any $1 \leq k \neq l \leq N$, define $U_{kl} \in \mathrm{SU}(2^N)$ by
        \begin{equation}\label{SWAP}
        U_{kl} := \exp\left( \frac{\pi i}{4} \sum_{\alpha=1}^3 \sigma_\alpha^{(k)} \sigma_\alpha^{(l)} \right),
        \end{equation}
        \noindent where $\sigma_\alpha^{(i)}$ is given by
        \begin{equation}\label{single_pauli_definition}
        \sigma_\alpha^{(i)} := \sigma_0 \otimes \cdots \otimes \sigma_0 \otimes \underbrace{\sigma_\alpha}_{i\text{-th}} \otimes \sigma_0 \otimes \cdots \otimes \sigma_0, \quad \alpha \in \{0, 1, 2, 3\}.
        \end{equation}

    \begin{lemma} \label{lem:UijSN}
     For any $I=(i_1, \ldots, i_N) \in \mathcal{I}_N^0$, 
    \begin{equation}\label{eq:SWAP}
    \mathrm{Ad}_{U_{kl}}(\tilde\sigma_I) = \tilde\sigma_{I'},
    \end{equation}
    where $I'$ is obtained from $I$ by switching $i_k$ and $i_l$.
    \end{lemma}

    \noindent
    {\bf Proof}:
     Note that for any $I=(i_1, \ldots, i_N) \in \mathcal{I}_N^0$, 
     $\tilde\sigma_I = -i \prod_{j=1}^N \sigma_{i_j}^{(j)}$. So
    \[ 
    \mathrm{Ad}_{U_{kl}}(\tilde\sigma_I) = -i \prod_{j=1}^N \mathrm{Ad}_{U_{kl}}(\sigma_{i_j}^{(j)}),
    \]
    where 
    $\mathrm{Ad}_{U_{kl}} (\sigma_{r}^{(j)}) :=  U_{kl} \, \sigma_{r}^{(j)} \, U_{kl}^{-1}$.   
    Hence to prove this lemma, we only need to show
    \begin{equation}\label{eq:SWAP1}
    \mathrm{Ad}_{U_{kl}}(\sigma_{r}^{(k)}) = \sigma_r^{(l)}, \hspace{20pt}
    \mathrm{Ad}_{U_{kl}}(\sigma_{r}^{(l)}) = \sigma_{r}^{(k)}, \hspace{20pt}
    \mathrm{Ad}_{U_{kl}}(\sigma_r^{(t)}) = \sigma_r^{(t)}
    \end{equation}
    for all $r \in \{0, 1,2,3\}$ and  $t \in \{1, \ldots,  N\} \setminus \{k, l\}$. 
    This is trivial if $r=0$. 
    By definition, $U_{kl}$ commutes with $\sigma_r^{(t)}$ if $t \neq k, l$.
    So the third  equality in \eqref{eq:SWAP1} is also trivial. 
    Since $U_{kl}$ is symmetric with respect to $k$ and $l$, we only need to show the first equality
    in \eqref{eq:SWAP1}. Moreover,  the right hand side of equation \eqref{SWAP} is
    symmetric with respect to  $\alpha \in \{1, 2, 3\}$. So
    it suffices to prove  the first equality in \eqref{eq:SWAP1} for $r = 1$.

    Let 
    \[ {\rm Id} := \sigma_0 \otimes \cdots \otimes \sigma_0, \hspace{20pt} 
      \Gamma_{kl}:= \sum_{\alpha=1}^3 \sigma_\alpha^{(k)} \sigma_\alpha^{(l)}.  \]
    Notice that the operators $i\sigma_1^{(k)}\sigma_1^{(l)}$, $i\sigma_2^{(k)}\sigma_2^{(l)}$, and $i\sigma_3^{(k)}\sigma_3^{(l)}$ commute with each other,  and
         $ \left( \sigma_\alpha^{(k)} \sigma_\alpha^{(l)} \right)^{2m} = {\rm Id} \,\,$
        for all integers $m \geq 0$. So we have
    \[
    U_{kl}= \prod_{\alpha=1}^3 \exp\left(\frac{\pi i}{4}\sigma_\alpha^{(k)}\sigma_\alpha^{(l)}\right)
        = \prod_{\alpha=1}^3 \left(\cos(\pi/4){\rm Id} + i \sin(\pi/4) \sigma_\alpha^{(k)}\sigma_\alpha^{(l)}\right)
        = \frac{\sqrt{2}(1+i)}{4} \left({\rm Id} + \Gamma_{kl} \right).
    \]
    A formula for $U_{kl}^{-1}$ is obtained if we replace $i$ by $-i$ in the above formula.
    Since
    \[ \Gamma_{kl} \, \sigma_1^{(k)} + \sigma_1^{(k)} \, \Gamma_{kl}  = 2 \sigma_1^{(l)}, \hspace{20pt}
        \Gamma_{kl} \, \sigma_1^{(k)}  \, \Gamma_{kl} =  - \sigma_1^{(k)} + 2 \sigma_1^{(l)},
    \]
    we have
    \[
    \mathrm{Ad}_{U_{kl}}(\sigma_1^{(k)}) 
     = \frac{1}{4} ({\rm Id} + \Gamma_{kl}) \, \sigma_1^{(k)}  \, ({\rm Id} + \Gamma_{kl})= \sigma_1^{(l)}.
    \]
    This finishes the proof of the lemma.
    $\Box$      
    
    \begin{remark}
    The operator $\mathrm{Ad}_{U_{kl}}$ is similar to the swap operator acting on $(\mathbb{C}^2)^{\otimes N}$
    (see, for example, \cite{NC}). 
    \end{remark}

        We now compute the normalizer ${\rm N}_\mathcal{G}(\mathcal{H})$ where $\mathcal{G}$ and $\mathcal{H}$
        are defined by equation \eqref{eqn:CalGH}.
        \begin{lemma}\label{split}
        \[
        {\rm N}_\mathcal{G}(\mathcal{H}) \cong ({\rm C}_\mathcal{G}(\mathcal{H}) \cdot \mathcal{H} ) \rtimes S_N.
        \]
        \end{lemma}

        \noindent
        {\bf Proof}:
        We prove this lemma by showing that the short exact sequence \eqref{shortexactsequence} splits.
        Let $\Omega$ be the subgroup of ${\rm Inn}(\mathfrak{su}(2^N))$ generated by $\{ \mathrm{Ad}_{U_{ij}} \mid 1 \leq i < j \leq N \}$. By Lemma~\ref{lem:UijSN}, $\Omega$ is the same 
        as $\overline{S}_N$ defined by equation \eqref{eqn:Sbar}. In particular, $\Omega$
        is isomorphic to the permutation group $S_N$. So there exists a group homomorphism
        $\tau: \Omega \to S_N$
         such that 
        for all $\omega \in \Omega$ and $I = (i_1, \, \ldots\, , \, i_N)  \in \mathcal{I}_N^0$, 
        \begin{equation} \label{eqn:P=SN}
        \omega(\tilde\sigma_I) = \tilde\sigma_{(i_{\tau(\omega)(1)}, \ldots, i_{\tau(\omega)(N)})}. 
        \end{equation}
        As noted in equation \eqref{eqn:SPhiInG}, 
        $\Omega = \overline{S}_N \subset \mathcal{G}$. 
         Moreover, any element in $\mathcal{H}$ can be written as 
        $\mathrm{Ad}_{U_1 \otimes \cdots \otimes U_N}$ with $U_i \in \mathrm{SU}(2)$.
        For any $\omega \in \Omega$, we have
        \begin{equation}\label{permutation}
        \omega \circ \mathrm{Ad}_{U_1 \otimes \cdots \otimes U_N} \circ \omega^{-1} 
            = \mathrm{Ad}_{U_{\tau(\omega)(1)} \otimes \cdots \otimes U_{\tau(\omega)(N)}}
            \in \mathcal{H}.
        \end{equation}
        Therefore, $\Omega \subset  {\rm N}_{\mathcal{G}}(\mathcal{H})$. 
        
        Now consider the map ${\rm N}_\mathcal{G} (\mathcal{H}) \to
        {\rm N}_\mathcal{G} (\mathcal{H})/({\rm C}_\mathcal{G} (\mathcal{H}) \cdot \mathcal{H})
        \subseteq \mathrm{Out}(\mathcal{H})$
        in the short exact sequence  \eqref{shortexactsequence}.
        Equation \eqref{permutation} implies that the image of $\Omega$ under this map  is isomorphic
        to $S_N$. By Lemmas \ref{lem:SU2SO3N} and \ref{lem:AutOutSO3N},  $\mathrm{Out}(\mathcal{H})$
        is also isomorphic to $S_N$. Hence
        \[ {\rm N}_\mathcal{G} (\mathcal{H})/({\rm C}_\mathcal{G} (\mathcal{H}) \cdot \mathcal{H}) = \mathrm{Out}(\mathcal{H})
        \cong S_N.\]
        The map $\tau^{-1}: S_N \to \Omega \subset  {\rm N}_{\mathcal{G}}(\mathcal{H}) $ then gives the 
        splitting map for the short exact sequence~\eqref{shortexactsequence}.
        It follows that ${\rm N}_\mathcal{G} (\mathcal{H})$ is a semidirect product
        of ${\rm C}_\mathcal{G} (\mathcal{H}) \cdot \mathcal{H}$ and $S_N$
        (see, for example, \cite[Exercise 11 on Page 26]{AlperinB}).        
        $\Box$

    We can now complete the proof of Theorem \ref{normalizer}.
    
        \noindent
        {\bf Proof of Theorem \ref{normalizer}}:

        Note that the center of $\mathrm{SO}(3)$ is trivial since $\mathrm{SO}(3)$  is a simple group.
        This implies that the center of $\mathcal{H} \cong \mathrm{SO}(3)^N$ is also trivial.
        Hence ${\rm C}_\mathcal{G} (\mathcal{H}) \cap \mathcal{H} = \{e\}$. Moreover,
        elements in ${\rm C}_\mathcal{G} (\mathcal{H})$ commute with elements in $\mathcal{H}$.
        So ${\rm C}_\mathcal{G} (\mathcal{H}) \cdot \mathcal{H} \cong \mathcal{H} \times {\rm C}_\mathcal{G} (\mathcal{H})$.
        Theorem \ref{normalizer} then follows from equation \eqref{eqn:IsomNormal} and Lemmas \ref{split}, 
        \ref{lem:SU2SO3N},  and \ref{centralizer}.
         $\Box$

\section{Calculation of the Isometry Group}

    In this section, we give a complete description of $\mathrm{Isom}(\mathrm{SU}(2^N), g)$
    for any penalty metric  $g$ satisfying Condition \eqref{cond:V1_invariant}. 
    We first give a complete description of its isotropy subgroup $\mathrm{Isom}_e(\mathrm{SU}(2^N), g)$
    which is isomorphic to a subgroup $\mathcal{L}(\mathrm{Isom}_e(\mathrm{SU}(2^N), g))$
    in $\mathrm{O}(\mathfrak{su}(2^N), g)$.
    Recall elements of $\mathcal{L}(\mathrm{Isom}_e(\mathrm{SU}(2^N), g))$ are just differentials of
    elements in $\mathrm{Isom}_e(\mathrm{SU}(2^N), g)$ at $e$.
    We have proved in Theorem \ref{normalizer} that 
    \[ \mathcal{L}(\mathrm{Isom}_e(\mathrm{SU}(2^N), g))  \subseteq
        \left( \,\mathcal{H} \times \Phi \, \right) \rtimes \overline{S}_N, \]
    where $\mathcal{H}$, $\overline{S}_N$ and $\Phi$ are defined by equations \eqref{eqn:CalGH}, \eqref{eqn:Sbar} and \eqref{eqn:Phi} respectively.
    We already know that  $\mathcal{H}$ is the identity component of $\mathcal{L}(\mathrm{Isom}_e(\mathrm{SU}(2^N), g))$. It is also easy to see the following:
    
    \begin{lemma} \label{lem:SNIsom}
    $\overline{S}_N \subset \mathcal{L}(\mathrm{Isom}_e(\mathrm{SU}(2^N), g)).$
    \end{lemma}
  
    \noindent
    {\bf Proof}:
    We have seen in the proof of Lemma \ref{split} that $\overline{S}_N = \Omega \subset
    \mathrm{Inn}(\mathfrak{su}(2^N))$ whose elements
     are the differentials of elements in $\mathrm{Inn}(\mathrm{SU}(2^N))$ at $e$. By Lemma \ref{automorphism_isometry} and equation~\eqref{eqn:SPhiInG},
    elements in $\overline{S}_N$ are the differentials of elements in $\mathrm{Isom}_e(\mathrm{SU}(2^N), g)$ 
    at $e$.
    The lemma is thus proved.
    $\Box$

    To determine $\mathcal{L}(\mathrm{Isom}_e(\mathrm{SU}(2^N), g))$, it remains to find which elements in $\Phi$
     lie in $\mathcal{L}(\mathrm{Isom}_e(\mathrm{SU}(2^N), g))$. 
     Using the fact that every isometry preserves the curvature tensor, we can exclude most elements
     in $\Phi$.
    
    \begin{lemma}\label{curvature_restriction}
    Let $g$  be a  penalty metric on $\mathrm{SU}(2^N)$. If an element
    $ \phi \in \Phi$  is the differential of an isometry with respect to $g$,  
    then there exist $c_{1}, c_{2} \in \{\pm 1\}$ such that
    \begin{equation} \label{eqn:phic1c2}
    \phi(\tilde\sigma_I)  
    = \begin{cases} c_{1} \, \tilde\sigma_I, &\textnormal{if } w(\tilde\sigma_I) \textnormal{ is odd,}\\
                    c_{2} \, \tilde\sigma_I, & \textnormal{if }w(\tilde\sigma_I) \textnormal{ is even,}
      \end{cases}
    \end{equation}
    for all $I \in \mathcal{I}_N^0$.
    \end{lemma}

    \noindent
    {\bf Proof}:
    Recall $\phi$ always has the form given by equation \eqref{eqn:phiepsilon}. For every $I \in \mathcal{I}_N^0$,
    there exists $s_I \in \{ \pm 1\}$ such that $\phi(\tilde\sigma_I) = s_I \tilde\sigma_I$. 
    Assume  $ \phi$  is the differential of an isometry with respect to 
    the metric $g$. Then $ \phi$  must preserve the curvature tensor $R$ corresponding to $g$.
    For any $I, J, K, L \in \mathcal{I}_N^0$, let $R_{IJKL}:= R(\tilde\sigma_I, \tilde\sigma_J, \tilde\sigma_K, \tilde\sigma_L)$.
    Then we must have
    \[ R_{IJKL}= R(\phi(\tilde\sigma_I), \phi(\tilde\sigma_J), 
                    \phi(\tilde\sigma_K), \phi(\tilde\sigma_L))
            = s_I s_J s_K s_L R_{IJKL}, \]
    where the last equality follows from the linearity of the curvature tensor. Hence we have
    \begin{equation} \label{eqn:4s=1}
     s_I s_J s_K s_L = 1 \hspace{20pt} {\rm if} \,\,\, R_{IJKL} \neq 0. 
    \end{equation}     
    Assume the penalty  factors of $g$ are  $q_1, \cdots, q_N$. Then $R_{IJKL}$
    is given by equation \eqref{curvature_tensor_formula}.       
    The lemma will be proved using equations \eqref{eqn:4s=1} and \eqref{curvature_tensor_formula}.
   We only need to show that for all $I, J \in \mathcal{I}_N^0$, $s_I = s_J$ 
    if $w(\tilde\sigma_I)$ and $w(\tilde\sigma_J)$ have the same parity.
    If we can find $K, L \in \mathcal{I}_N^0$ such that
    $s_K=s_L$ and $R_{IJKL} \neq 0$, then equations \eqref{eqn:4s=1} implies
    $s_I s_J s_K^2=1$, which in turn implies $s_I=s_J$.
    Recall for any $1 \leq r_1 < \cdots < r_k \leq N$, the space $V_{(r_1, \ldots, r_k)}$
    is defined by equation \eqref{def_irr}. Note that $V_{(r_1, \ldots, r_k)}$
    is still well defined if $r_1, \ldots, r_k$ are not in the increasing order. 
    By definition of $\phi$, $s_K = s_L$ if $\tilde\sigma_K, \tilde\sigma_L \in V_{(r_1, \ldots, r_k)}$
    for some $r_1, \ldots, r_k$. We will prove the lemma in the following four steps.

    {\bf Step 1}:  Prove  $s_I = s_J$ if  $w(\tilde\sigma_I) = w(\tilde\sigma_J) = 1$.

    In this case, there exist $r_1, r_2 \in \{1, \ldots, N\}$ such that
    $\tilde\sigma_I \in V_{(r_1)}$ and $\tilde\sigma_J \in V_{(r_2)}$.
    If $r_1=r_2$, $s_I = s_J$ holds trivially. So we only need to consider
    the case where $r_1 \neq r_2$. Without loss of generality, we may assume
    $r_1 < r_2$. Since $s_I = s_{I'}$ and $s_J = s_{J'}$ whenever
    $\tilde\sigma_{I'} \in V_{(r_1)}$ and $\tilde\sigma_{J'} \in V_{(r_2)}$.
    So we only need to prove  $s_I = s_J$ for some particular choice of
    $\tilde\sigma_I \in V_{(r_1)}$ and $\tilde\sigma_J \in V_{(r_2)}$.    
    It is convenient to use the notation $\sigma_k^{(\alpha)}$ defined in equation~\eqref{single_pauli_definition}. 
    Choose
    \[
    \tilde\sigma_I = -i \, \sigma^{(r_1)}_1, \hspace{15pt} 
    \tilde\sigma_J = -i \, \sigma^{(r_2)}_2, \hspace{15pt}
    \tilde\sigma_K = -i \, \sigma^{(r_1)}_2\sigma_1^{(r_2)}, \hspace{15pt}
    \tilde\sigma_L = -i \, \sigma_3^{(r_1)}\sigma_3^{(r_2)}.
    \]
   Then  $\tilde\sigma_I \in V_{(r_1)}$, $\tilde\sigma_J \in V_{(r_2)}$, and
   $\tilde\sigma_K, \tilde\sigma_L \in V_{(r_1, r_2)}$.   
    A straightforward computation using equation \eqref{curvature_tensor_formula} shows
    $R_{IKJL} = -q_1^2 / q_2 \neq 0.$
    Since $s_K = s_L$, by equation 
     \eqref{eqn:4s=1},  we have  $s_I = s_J$.

    {\bf Step 2}: Prove  $s_I = s_J$ if $w(\tilde\sigma_I) = w(\tilde\sigma_J) = 2$.
    
    In this case, there exists $1 \leq a_1 < a_2 \leq N$ and $1 \leq b_1 < b_2 \leq N$ such
    that  $\tilde\sigma_I \in V_{(a_1, a_2)}$ and $\tilde\sigma_J \in V_{(b_1, b_2)}$.
    If $\{ a_1, a_2\} = \{b_1, b_2\}$, then $s_I = s_J$  holds trivially.
    It suffices to show that $s_I = s_J$ holds for the case where $\{ a_1, a_2\} \cap \{b_1, b_2\}$ consists
    of only one element.
    This is because if $\{ a_1, a_2\} \cap \{b_1, b_2\}$ is empty, we can choose
     an arbitrary $\tilde\sigma_K \in V_{(a_1, b_1)}$.
    We can prove $s_I=s_J$ by showing $s_I = s_K$ and $s_K = s_J$.

    Without loss of generality, we may assume $a_1=b_1$ and $a_2 \neq b_2$. As in Step 1, it suffices to
    prove $s_I=s_J$ for any particular choice of  $\tilde\sigma_I \in V_{(a_1, a_2)}$ and $\tilde\sigma_J \in V_{(b_1, b_2)}$.
   Choose
    \[
    \begin{aligned}
    \tilde\sigma_I &= -i \, \sigma_1^{(a_1)}\sigma_1^{(a_2)}, & 
    \tilde\sigma_J &= -i \, \sigma_2^{(a_1)}\sigma_1^{(b_2)}, \\
    \tilde\sigma_K &= -i \, \sigma_1^{(a_1)}\sigma_2^{(a_2)}\sigma_2^{(b_2)}, & 
    \tilde\sigma_L &= -i \, \sigma_2^{(a_1)}\sigma_3^{(a_2)}\sigma_3^{(b_2)}.
    \end{aligned}
    \]
    Then $s_K=s_L$ since $\tilde\sigma_K, \tilde\sigma_L \in V_{(a_1, a_2, b_2)}$.
    A straightforward computation using equation~\eqref{curvature_tensor_formula} shows
    $ R_{IKJL} = q_3 \neq 0.$
   By equation \eqref{eqn:4s=1}, we have $s_I = s_J$.

    {\bf Step 3}: Prove the {\bf claim}: 
    For any $J \in \mathcal{I}_N^0$ with $w(\tilde\sigma_J) \geq 3$, we can find $I \in \mathcal{I}_N^0$,
    such that $w(\tilde\sigma_I) = w(\tilde\sigma_J) - 2$ and $s_I = s_J$.

    Assume $w(\tilde\sigma_J) = k+2$ with $k \geq 1$ and $\tilde\sigma_J \in V_{(r_1, \ldots, r_{k+2})}$ for some
    $1 \leq r_1 < \ldots < r_{k+2} \leq N$.
    Since $s_{J'} = s_{J}$ for all  $\tilde\sigma_{J'} \in V_{(r_1, \ldots, r_{k+2})}$, we may only consider
    one particular $\tilde\sigma_J \in V_{(r_1, \ldots, r_{k+2})}$.

    If $k = 1$, we can choose $\tilde\sigma_J = -i \, \sigma_2^{(r_1)}\sigma_3^{(r_2)}\sigma_1^{(r_3)}$ and
     \[
     \tilde\sigma_I = -i \, \sigma_1^{(r_1)},  \hspace{20pt} 
     \tilde\sigma_K = -i \, \sigma_2^{(r_1)}\sigma_3^{(r_2)}, \hspace{20pt}
     \tilde\sigma_L = -i \, \sigma_1^{(r_1)}\sigma_1^{(r_3)}.
     \]
     By Step 2, $s_K=s_L$.  A straightforward computation using equation~\eqref{curvature_tensor_formula} shows
     $  R_{IKJL} = q_1 \neq 0$.  By equation \eqref{eqn:4s=1}, we have
     $s_I=s_J$.
     It follows that  the claim holds for all $J$ with $w(\tilde\sigma_J) = 3$. By the result of Step 1, we also have
     $s_I=s_J$ whenever  $w(\tilde\sigma_I) = 1$ and $w(\tilde\sigma_J) = 3$.

    If $k > 1$, we can choose
     $ \tilde\sigma_J = -i \, \sigma_1^{(r_1)}\cdots \sigma_1^{(r_{k - 1})}\sigma_3^{(r_k)}
                \sigma_3^{(r_{k + 1})}\sigma_1^{(r_{k + 2})} $
    and
    \[
    \tilde\sigma_I = -i \, \sigma_1^{(r_1)} \cdots \sigma_1^{(r_{k - 1})}\sigma_2^{(r_k)}, \hspace{20pt}
    \tilde\sigma_K = -i \, \sigma_3^{(r_k)}, \hspace{20pt}
    \tilde\sigma_L = -i \, \sigma_2^{(r_k)}\sigma_3^{(r_{k + 1})}\sigma_1^{(r_{k + 2})}.
    \]
    By the result of the last paragraph, $s_K=s_L$.  
    A straightforward computation using equation~\eqref{curvature_tensor_formula} shows
    \[
    R_{IJKL} = \frac{q_1(q_k + q_{k + 2} - q_3)}{q_k},  \hspace{20pt}
    R_{IKLJ} = -\frac{q_1(q_3 + q_{k + 2} - q_k)}{q_3}.
    \]
    If both $R_{IJKL} = 0$ and $R_{IKLJ} = 0$, we would obtain $q_{k + 2} = 0$, which contradicts the positivity of the metric. Consequently, either $R_{IJKL} \neq 0$ or $R_{IKLJ} \neq 0$.  Hence we can still use equation \eqref{eqn:4s=1} to
    obtain $s_I = s_J$. The claim is thus proved.

    {\bf Step 4}: By the claim in Step 3 and an induction on weight, we can show that
    for any $\tilde\sigma_J$, there exists $\tilde\sigma_I$ with $w(\tilde\sigma_I) \in \{1, 2\}$ such that
    $s_I=s_J$. The lemma then follows from Steps 1 and 2.
     $\Box$

    We can exclude more elements in $\Phi$ from being an element in $\mathcal{L}(\mathrm{Isom}_e(\mathrm{SU}(2^N), g))$
    using classification of symmetric spaces.
    
    \begin{lemma}\label{lem:symmetric_condition}
     Let $g$  be a  penalty metric on $\mathrm{SU}(2^N)$ which is not bi-invariant. If 
    $\phi \in \Phi$  is the differential of an isometry with respect to $g$,  
    then $\phi = {\rm id}$ or $\varphi_0$, where $\varphi_0$ is defined by
    \begin{equation} \label{eqn:varphi0}
    \varphi_0(\tilde\sigma_I)  
    := (-1)^{w(\tilde\sigma_I)+1} \tilde\sigma_I
    \end{equation}
    for all $I \in \mathcal{I}_N^0$. On the other hand, $\varphi_0$ is indeed the differential of an isometry.
    \end{lemma}

    \noindent
    {\bf Proof}:
    By Lemma \ref{curvature_restriction}, we only need to consider $\phi$ given by equation \eqref{eqn:phic1c2},
    i.e. $\phi= \pm {\rm id}$, or $\phi = \pm \varphi_0$. We first show that $\varphi_0$ is the differential of
    an isometry by representing $\varphi_0$ as a composition of the differentials of two isometries.  
    
    Let $F: \mathrm{SU}(2^N) \to \mathrm{SU}(2^N)$ be the automorphism which maps $U \in \mathrm{SU}(2^N)$ to 
    its complex conjugate $\overline{U}$. 
    Let $f = F_{* e}: \mathfrak{su}(2^N) \to \mathfrak{su}(2^N)$ be the differential of $F$ at $e$.
    Then $f(A) = \overline{A}$ for all $A \in \mathfrak{su}(2^N)$. In particular, for all $I=(i_1, \ldots, i_N) \in \mathcal{I}_N^0$,
    \[ f(\tilde\sigma_I) = (-1)^{\mu(I)+1}\tilde\sigma_I,\]
     where $\mu(I)$ is the number of $k$ such that $i_k=2$.
     It follows that $f$ preserves any penalty metric $g$. By Lemma \ref{automorphism_isometry}, $F$ is an isometry
     with respect to $g$.
     
     Note that $i \sigma_2 \in \mathrm{SU}(2)$. Let 
     $U_0 := (i \sigma_2 )^{\otimes N} = i^N \sigma_2^{\otimes N} \in \mathrm{SU}(2)^{\otimes N}$.
    By Proposition \ref{prop:local_unitary_isometry}, $\mathrm{Ad}_{U_0}$ is the differential of
    an isometry $\mathcal{A}\mathrm{d}_{U_0}$. For any $I=(i_1, \ldots, i_N) \in \mathcal{I}_N^0$,
    \[
    \mathrm{Ad}_{U_0}(\tilde\sigma_I) = (-1)^{\eta(I)}\tilde\sigma_I,
    \]
    \noindent where $\eta(I)$ is the number of $k$ such that $i_k \in \{1, 3\}$.
    Since $\mu(I) + \eta(I) = w(\tilde\sigma_I)$, we have 
    \begin{equation} \label{eqn:phi0f}
     \varphi_0 = f \circ \mathrm{Ad}_{U_0}. 
    \end{equation}
    This proves that $\varphi_0$ is the differential of an isometry on $\mathrm{SU}(2^N)$.
   
   We now show that $-{\rm id}$ or $- \varphi_0$ cannot be the differentials of isometries.
    In fact if $- {\rm id}$ is the differential of isometry, then the geodesic symmetry with respect to $e$ is an isometry. 
    Since $g$ is right invariant, this implies that the geodesic symmetry with respect to every point must be isometries. 
    Therefore  $(\mathrm{SU}(2^N), g)$ is a Riemannian symmetric space. By the classification of compact simply connected 
    irreducible symmetric spaces (see, for example, \cite[Section 6 in Chapter X]{Helgason}), 
    $g$ must be bi-invariant. This contradicts to the assumption of $g$.  
    Hence $- {\rm id}$ cannot be the differential of isometry.
    Since $- {\rm id} = (- \varphi_0 ) \circ \varphi_0$ and $\varphi_0$ is the differential of an isometry,
    this also implies that $- \varphi_0$ cannot be the differential of isometry.
    The lemma is thus proved.
     $\Box$

    Now we can give a complete description of $\mathcal{L}(\mathrm{Isom}_e(\mathrm{SU}(2^N), g))$.
     
    \begin{theorem}
    \label{mainthm}
     Assume $N \geq 2$. For any  penalty metric $g$ on $\mathrm{SU}(2^N)$ satisfying Condition \eqref{cond:V1_invariant}, we have    \begin{equation}\label{eq:isotropy_group}
    \mathcal{L}(\mathrm{Isom}_e(\mathrm{SU}(2^N), g)) 
    = (\overline{\mathrm{Ad}}(\mathrm{SU}(2)^{\otimes N}) \rtimes \{f,\mathrm{id}\}) \rtimes \overline{S}_N,
    \end{equation}
    where $f: \mathfrak{su}(2^N) \to \mathfrak{su}(2^N)$ is the map taking complex conjugation, and 
    $\overline{S}_N$ is defined by equation \eqref{eqn:Sbar}.
    \end{theorem}

    \noindent
    {\bf Proof}: By Theorem \ref{normalizer}, Proposition \ref{prop:local_unitary_isometry}, 
     and Lemmas \ref{lem:SNIsom} and \ref{lem:symmetric_condition}, we have
    \[ \mathcal{L}(\mathrm{Isom}_e(\mathrm{SU}(2^N), g)) 
    = (\overline{\mathrm{Ad}}(\mathrm{SU}(2)^{\otimes N}) \rtimes \{\varphi_0,\mathrm{id}\}) \rtimes \overline{S}_N, \]
    where $\varphi_0$ is defined by equation \eqref{eqn:varphi0}.   By equation \eqref{eqn:phi0f}, $\varphi_0$ is a composition
    of $f$ and an element in $\overline{\mathrm{Ad}}(\mathrm{SU}(2)^{\otimes N})$. Therefore we can replace
    $\varphi_0$ by $f$ and obtain the desired result.
    $\Box$

   By Theorem \ref{mainthm}, $\mathrm{Isom}_e(\mathrm{SU}(2^N), g) \subset {\rm Aut}(\mathrm{SU}(2^N))$. Hence, by Lemma \ref{normalsubgp}, 
   we obtain the following result for the full isometry group:
 
    \begin{corollary} \label{cor:mainResult}
    Assume $N \geq 2$. For any  penalty metric $g$ on $\mathrm{SU}(2^N)$ satisfying Condition \eqref{cond:V1_invariant},
   \begin{equation}\label{eq:full_isometry_group}
    \mathrm{Isom}(\mathrm{SU}(2^N), g) = R(\mathrm{SU}(2^N)) \rtimes 
            ((\overline{\mathcal{A}\mathrm{d}}(\mathrm{SU}(2)^{\otimes N}) \rtimes \{F,\mathrm{id}\}) \rtimes \overline{\mathcal{S}}_N),
    \end{equation}
    where $R(\mathrm{SU}(2^N))$ is the group of all right translations by elements in $\mathrm{SU}(2^N)$, 
          $F$ is the map taking complex conjugation,
         $\overline{\mathcal{S}}_N$ is the group generated by $\mathcal{A}\mathrm{d}_{U_{ij}}$
         for all $1 \leq i < j \leq N$, and $U_{ij}$ is defined by equation \eqref{SWAP}.
    \end{corollary}

   Together with Lemma \ref{lem:SU2SO3N},  this result implies Theorem \ref{thm:mainIntro}.

\vspace{30pt} \noindent
Xiaobo Liu \\
School of Mathematical Sciences \& \\
Beijing International Center for Mathematical Research, \\
Peking University, Beijing, China. \\
Email: {\it xbliu@math.pku.edu.cn}
\ \\ \ \\
Lei Zheng \\
School of Mathematical Sciences, \\
Peking University, Beijing, China. \\
Email: {\it zhenglei0813@stu.pku.edu.cn}

\end{document}